\documentclass[12pt, a4paper, reqno]{amsart}

\usepackage{amsmath, amssymb, amsthm}
\usepackage{mathtools}
\usepackage[alphabetic]{amsrefs}
\usepackage{xcolor}
\usepackage{url}
\usepackage[mathscr]{euscript}
\usepackage{enumitem}
\usepackage[colorlinks=true, linkcolor= blue, citecolor=magenta]{hyperref}
\usepackage[capitalize, noabbrev]{cleveref}
\usepackage{NichsCustom}
\usepackage{dsfont}
\usepackage{bbm}
\usepackage{kbordermatrix}
\usepackage[margin=3cm]{geometry}
\DeclareRobustCommand{\qedify}[1]{%
  \ifmmode \quad\hbox{#1}
  \else
    \leavevmode\unskip\penalty9999 \hbox{}\nobreak\hfill
    \quad\hbox{#1}%
  \fi
}
\newenvironment{example}{\begin{example*}\pushQED{\qedify{$\diamondsuit$}}}{\popQED\end{example*}}

\numberwithin{equation}{section}

\newcommand{\sV}{\mathcal{V}}

\newcommand{\bfx}{\mathbf{x}}

\newcommand{\bfu}{\mathbf{u}}
\newcommand{\bfv}{\mathbf{v}}

\DeclareMathOperator{\rank}{rank}

\DeclareMathOperator{\GL}{GL}

\DeclareMathOperator{\uMat}{\underline{Mat}}

\DeclareMathOperator{\Star}{star}

\makeatletter
\newcommand{\superimpose}[2]{{\ooalign{$#1\@firstoftwo#2$\cr\hfil$#1\@secondoftwo#2$\hfil\cr}}}
\makeatother
\newcommand{\ttimes}{\hspace{0.3mm}{\mathpalette\superimpose{{\circ}{\cdot}}}\hspace{0.3mm}}

\begin{document}

\title{Matroid Products in Tropical Geometry}

\author{Nicholas Anderson}
\address{School of Mathematical Sciences, Queen Mary University of London, Mile End Road, London E1 4NS, United Kingdom}
\email{n.c.anderson@qmul.ac.uk}
%\thanks{}

\begin{abstract}
Symmetric powers of matroids were first introduced by Lovasz \cite{Lovasz77} and Mason \cite{MGlueing} in the 1970s, where it was shown that not all matroids admit higher symmetric powers. Since these initial findings, the study of matroid symmetric powers has remained largely unexplored. In this paper, we establish an equivalence between valuated matroids with arbitrarily large symmetric powers and tropical linear spaces that appear as the variety of a tropical ideal. In establishing this equivalence, we additionally show that all tropical linear spaces are connected through codimension one. These results provide additional geometric and algebraic connections to the study of matroid symmetric powers, which we leverage to prove that the class of matroids with second symmetric power is minor closed and has infinitely many forbidden minors.

\commentout{FOR USE IN PUBLISHED VERSION:
\smallskip
\noindent \textbf{Data Availability Statement.} Data sharing not applicable to this article as no datasets were generated or analysed during the current study.

\smallskip
\noindent \textbf{Declarations} The author is funded by Queen Mary University of London through the Queen Mary Principal's Award.

 }
\end{abstract}

\maketitle

\section{Introduction}\label{sec:intro}

In order to approach this problem, we revisit some early works of Mason \cite{MGlueing} and Las Vergnas \cite{LV81} on symmetric and tensor products of matroids. Both authors show that matroid powers do not necessarily exist. In the case of LasVergnas, it is shown that nontrivial tensor products with the V\'amos matroid never exist, and for Mason a specific point configuration is shown to lack large symmetric powers. Moreover, we can deduce from the nonexistence of matroid tensor powers the non-existence of second matroid symmetric powers; specifically, since the V\'amos matroid $V_8$ has no tensor powers, $V_8\oplus V_8$ has no second symmetric power.

We will follow the min-plus convention for the tropical semiring, which is the set $\T:= \R\cup \set{\infty}$ together with operations $\oplus$ and $\ttimes$ with $a\oplus b= \min(a,b)$ and $a\ttimes b = a+b$. Our approach in this paper naturally extends the work of Mason to the context of \textbf{valuated matroids}, which are subsemimodules $\sV\subset \T^n$ satisfying the following \textbf{vector elimination axiom}
    \begin{quote}
        For any vectors $G,H\in \sV$ and any $e\in [n]$ with $G_e=H_e$, there exists a vector $F\in \sV$ such that $F_e=\infty$, and $F_{e'}\geq\min{G_{e'},H_{e'}}$ for all $e'\neq e$ and with equality whenever $G_{e'}\neq H_{e'}$.
    \end{quote}
A \textbf{matroid} presented by its vectors is analogous to a valuated matroid for which the subsemimodule is generated by vectors $G$ such that $G_e\in \set{0,\infty}$ for all $e\in [n]$.

Throughout the paper we will make use of classical matroid theory and cryptomorphisms such as those for bases, circuits, spanning sets, flats, and hyperplanes as well as standard constructions such as taking minors and duals. We refer the reader to \cite{Oxley06} for a comprehensive background. We will define any relevant analogues for valuated matroids, such as the basis cryptomorphism and minors in \coloref{sec:LinTropIdeals}.

In this paper we will be primarily concerned with valuated matroid symmetric powers. Before providing this definition, we should discuss some useful notation. If $E$ is any set, the $d\th$ symmetric power of $E$, which we denote $\Sym_d(E)$, is the set of unordered $d$-tuples of elements of $E$. The motivation for defining $\Sym_d(E)$ comes from commutative algebra, where the set of degree $d$ monomials in $\set{x_1,\dots,x_n}$, denoted $\Mon_d(n)$, is precisely the set $\Sym_d(\set{x_1,\dots,x_n})$.

If we fix an element $\alpha \in \Sym_i(E)$, there is an injective ``multiplication" map
    \begin{align*}
        *_\alpha: \Sym_d(E)&\to \Sym_{d+i}(E)\\
        \beta&\mapsto (\alpha,\beta).
    \end{align*}
Since the coordinates of elements of $\Sym_{d+i}$ are unordered, we will write $(\alpha,\beta)$ as $\alpha\beta$, taking after commutative multiplication notation. 

We will often consider the subset of elements of $\Sym_{d+i}(E)$ that are divisible by some $\alpha\in \Sym_{i}(E)$, which we denote as

$$\Sym_{d+i}(E)|_{\alpha} := \set{e\in \Sym_{d+i}(E): e= \alpha x}.$$

Given a valuated matroid $\sV$ on $\Sym_{d+i}(E)$ we further extend this notation and write $\sV|_{\alpha}$ instead of $\sV|_{\Sym_{d+1}(E)|_{\alpha}},$ where the restriction of $\sV$ to the set $T\subset [n]$ is given by the subsemimodule 
$$S|_T:=\sV\cap \set{\pi_T(x)\in \T^T : x_i=\infty \text{ for all } i\notin T}.$$

\begin{definition}\label{def:introSymprod}
Fixing a matroid $\sV$ on a ground set $E$, a $\mathbf{d}^{\text{\textbf{th}}}$ \textbf{symmetric quasi power} of $\sV$ is a sequence of matroids $(\sV_1,\dots,\sV_d)$ with $\sV_1 = \sV$ and for all $1\leq i \leq d$
\begin{itemize}
    \item $E(\sV_i) = \Sym_i(E),$
    \item for all $\alpha\in \Sym_{d-i}(E)$ for which $\alpha$ is not divisible by a loop the multiplication map $*_\alpha$ given by $y\mapsto \alpha x$ induces the isomorphism
    \begin{equation} 
        \sV_i\cong \sV_d|_{\alpha}
    \end{equation}
    otherwise
    \begin{equation}
        \rk(\sV_d|_{\alpha})=0.
    \end{equation}
    
\end{itemize} 
Moreover, if $(\sV_1,\dots,\sV_d)$ is a symmetric quasi power of $\sV$ with

\begin{equation}
\rk(\sV_d)= \binom{\rk(\sV)+d-1}{d}
\end{equation}

then we say that $(\sV_1,\dots,\sV_d)$ is a $\mathbf{d}^{\text{\textbf{th}}}$ \textbf{symmetric power} of $\sV$.
\end{definition}

In \coloref{sec:MatProds} we will discuss relevant background on tensor and symmetric powers, establish geometric connections, and address how our definition differs from that of Mason in \cite{MGlueing}. We further prove the following theorem.

\begin{theorem}\label{thm:introThm0}
    The family of matroids with $d\th$ symmetric power is minor closed and has infinitely many forbidden minors.
\end{theorem}

\coloref{thm:introThm0} partly relies on the connection we make between symmetric powers and the geometry of their underlying matroids. To establish this connection we use the theory of tropical ideals.

\begin{definition}\label{def:tropicalIdeals}
    Let $I\subset \T[x_1,\dots,x_n]$ be a homogeneous ideal. Then $I$ is called a \textbf{tropical ideal} if $I_d:=I|_{\Mon_d(n)}$ is a valuated matroid in $\T^{\Mon_d(n)}$ for all $d\in \N$. 
\end{definition}

The most simple examples of tropical ideals are \textbf{realizable tropical ideals}, which are the tropicalizations of polynomial ideals $\trop(J)$, where $J$ is an ideal in $k[x_1,\dots,x_n]$ for some non-archimedean valued field $k$. We refer the reader to \cite{MR18} and \cite{MR20} for explicit examples, extensions to other polynomial semirings, and for a comprehensive background.

The underlying geometry of matroid symmetric powers and of tropical ideals is vital to establishing their connection. If $p\in \T[x_1,\dots,x_n]$ is a polynomial, its variety is given by
$$V(p)=\set{x\in \T^n: p \text{ attains its minimum at least twice at } x \text{ or } p(x)=\infty},$$
which is a balanced polyhedral complex in $\T^n$. The variety of a tropical ideal $I\subset \T[x_1,\dots,x_n]$ is then given by
$$V(I)=\bigcap_{p\in I}V(p),$$
which is also a balanced polyhedral complex, with weights defined as in \cite{MR20}. In this paper we refer to any balanced polyhedral complex in $\T^n$ as a \textbf{tropical variety}, and we distinguish those tropical varieties that arise as the variety of a tropical ideal by referring to them as \textbf{tropically realizable varieties}.

To relate tropical geometry explicitly to valuated matroids, we consider that $\T^n\cong \T[x_1,\dots,x_n]_1$ under the map $e_i\mapsto x_i$. In particular, a valuated matroid $\sV\subset \T^n$ can be seen as a collection of polynomials in $\T[x_1,\dots,x_n].$ Explicitly, given a vector $v\in \sV$ we have the associated vector-polynomial
$$p_v = \bigoplus_{i\in [n]} v_ix_i\in \T[x_1,\dots,x_n].$$
The \textbf{tropical linear space} associated to $\sV$ is then given by
$$\trop(\sV) := \bigcap_{v\in \sV}V(p_v),$$
which is a balanced polyhedral complex in $\T^n$ with constant weights $1$.

\begin{theorem}\label{thm:introThm1}
Let $\sV$ be a matroid. The tropical linear space $\trop(\sV)$ is the variety of a tropical ideal if and only if a $d\th$ symmetric power of $\sV$ exists for every $d\in \N$.
\end{theorem}

If the variety $\trop(\sV)$ is of the form $V(I)$ for a tropical ideal $I$, as in \coloref{thm:introThm1}, then $\trop(\sV)$ is a tropically realizable variety by definition. We extend this terminology to the matroid $\sV$ itself and say that $\sV$ is a \textbf{tropically realizable matroid} whenever $\trop(\sV)$ is a tropically realizable variety. Tropically realizable varieties behave nicely under coordinate projections and stable intersections with hyperplanes \cite{MR20}, which generalize the standard matroid minor operations of deletion and contraction, respectively, to tropical ideals. We will further see that, in the case of tropically realizable matroids, these geometric operation extend naturally to matroid minor operations on symmetric powers.

\begin{theorem}\label{thm:introThm2}
Tropically realizable matroids are a minor closed class. In particular, if $\sV$ is a matroid on $E$ and $(\sV_1,\sV_2,\dots,\sV_d)$ is a $d\th$ symmetric power of $\sV$, then for any minor $\sV'$ of $\sV$, there is a minor $\sV'_i$ of each $\sV_i$ for which $(\sV',\sV'_2,\dots, \sV'_d)$ is a $d\th$ symmetric power of $\sV'$.
\end{theorem}

The geometric techniques used to prove \coloref{thm:introThm2} extend to other classes of matroids. We say that the matroid $\sV$ is \textbf{tropically set-realizable} if there is a tropical ideal $I$ for which the support of $V(I)$ is equal to that of $\trop(\sV)$, but the weights are not necessarily equal to $1$. Tropically set-realizable matroids were originally discussed in \cite{Yu17}, and here the the proof of \coloref{thm:introThm2} immediately implies that this set of matroids is minor closed. 

We also consider the class of tropically algebraic matroids. Given a (tropical) ideal $I$ in $\T[x_1,\dots,x_n]$, we define the independence complex of $I$ to be the collection of subsets of $[n]$ given by

$$\mathcal{I}(I):= \set{S\subset [n]: I\cap \T[x_i: i\in S]= \set{\infty}}.$$
The independence complex of $I$ is always an abstract simplicial complex, but it is not always a matroid, even if we restrict to tropical ideals $I$. We say that a matroid $M$ is \textbf{tropically algebraic} if there is a tropical ideal $I$ for which $\mathcal{I}(I)$ is the collection of independent sets of $M$. In \coloref{sec:LinTropIdeals} we will prove the set of tropically algebraic matroids is minor closed as well. 

Underlying these aforementioned proofs are several technical results that are quite useful, and in particular we'd highlight a special case of \coloref{prop:connectedCodim1}. If $\Sigma$ is a polyhedral complex then we can define its \textbf{facet-ridge} graph $G(\Sigma)$. This is the graph whose vertices are given by the facets (top-dimensional cells) of $\Sigma$ and where two vertices share an edge if there is a ridge (codimension-1 face) between the associated facets in $\Sigma$. The polyhedral complex $\Sigma$ is said to be connected through codimension one if $G(\Sigma)$ is a connected graph.

\begin{corollary}\label{thm:introThm3}
Tropical linear spaces are connected through codimension 1.
\end{corollary}

\coloref{thm:introThm3} is vital to the proof of \coloref{thm:introThm1}, and has implications well beyond the scope of this paper.

\section*{Acknowledgements}
I'd like to give some special thanks to Felipe Rinc\'on and Alex Fink for their valuable insight and guidance with this paper, as well as to Josephine Yu, Diane Maclagan, and Zach Walsh for helpful conversations that inspired multiple improvements to the results in this paper.
%%%%%%%%%%%%%%%%%%%%%%%%%%%%%%%%%%%%%%%%%%%%%%%%%%%%

%%%%%%%%%%%%%%%%%%%%%%%%%%%%%%%%%%%%%%%%%%%%%%%%%%%%

%%%%%%%%%%%%%%%%%%%%%%%%%%%%%%%%%%%%%%%%%%%%%%%%%%%%

    %%%%%%%%%%%%%%%%%%%%%%%%%%%%%%%%%%%%%%%%%%%%%%%%%%%%%%%%%%%

\section{Matroid Products}\label{sec:MatProds}

In this section we will revisit some ideas of Mason \cite{MGlueing} and Las Vergnas \cite{LV81} particularly with respect to tensor and symmetric powers of matroids, working with matroids rather than valuated matroids.
   
Work with matroid products in the category of matroids with strong maps was explored by Crapo and Rota in \cite{CR70}, and it was known at the time of these works that matroid products did not exist in this category. However, this is not the notion of matroid product that we are discussing in this paper. Instead, we interest ourselves in matroid products in a weaker sense, as defined by Las Vergnas for tensor products, and by Mason for symmetric powers.
\subsection{Tensor Products}\label{subsec:tensorProducts}
    Given matroids $M$ and $N$, we aim to construct a new matroid, called a tensor product of $M$ and $N$, which will be a matroid with ground set equal to the Cartesian product $E(M)\times E(N)$ with rank $\rk(M)\rk(N)$ and inheriting some structure from $M$ and $N$. While this operation is always possible for matroids represented over a common field $k$, it is not necessarily unique, nor is it  possible for matroids in general as shown by LasVergnas in \cite{LV81}*{Proposition 2.1}. 
    
    Several results presented in this section have been known since LasVergnas's original paper, however, we revisit this work and introduce a new proof of \cite{LV81}*{Proposition 2.1}. Furthermore, our proof generalizes to an infinite class of non-algebraic matroids with no common non-realizable minors. This family of matroids is exactly the class of matroids that were initially studied by Lindstr\"om in \cite{Lind88}. We substantiate the claim that none of these matroids has a non-trivial matroid tensor product.
    
    \begin{definition}\label{def:tensorproduct}
    Let $M$ and $N$ be matroids on ground sets $E(M)$ and $E(N)$ respectively. A matroid $P$ on ground set $E(M)\times E(N)$ is called a \textbf{quasi product} of $M$ and $N$ if for every non-loop $e\in E(M)$ we have
    \begin{equation} \label{eqn:tensorRight}
        N\cong P|_{\set{(e,x)\mid x\in E(N)}}
    \end{equation}
    and for every non-loop $f\in E(N)$ we have
    \begin{equation}\label{eqn:tensorLeft}
        M\cong P|_{\set{(x,f)\mid x\in E(M)}},
    \end{equation}
    where these isomorphisms are explicitly given by the maps $x\mapsto (e,x)$ and $x\mapsto (x,f)$ respectively. If $e$ is a loop in $E(M)$ and $f$ a loop in $E(N)$ then we have
    \begin{equation}\label{eqn:tensorRight0}
        \rk(P|_{\set{(e,x)\mid x\in E(N)}})=0
    \end{equation}
    or
    \begin{equation}\label{eqn:tensorLeft0}
    \rk(P|_{\set{(x,f)\mid x\in E(M)}})=0
    \end{equation}
    respectively.

    We will refer to the set $\set{(e,x)\mid x\in E(N)}$ as the \textbf{row} associated to $e\in E(M)$ and we will refer to the set $\set{(x,f)\mid x\in E(M)}$ as the \textbf{column} associated to an element $f\in E(N)$ .
    \end{definition}
    
    \colorefiso{eqn:tensorRight} and \colorefiso{eqn:tensorLeft} imply that, given bases $B_1$ of $M$ and $B_2$ of $N$, that $B_1\times B_2$ is a spanning set of any quasi product of $M$ and $N$. Therefore the rank of a quasi product is bounded above by $\rk(M)\rk(N)$, and this bound can be obtained whenever $M$ and $N$ are $k$-realizable matroids. We call that a quasi product of $M$ and $N$ with rank $\rk(M)\rk(N)$ a \textbf{tensor product} of $M$ and $N$. Moreover, if $M$ and $N$ are the same matroid, then we refer to any of their tensor products as a \textbf{second tensor power}.

    The motivating example for the definition of matroid tensor products is, as is often the case, linear spaces. Given two matroids $M_1$ and $M_2$ that are represented as linear subspaces of $L(M_1)\subset k^{d_1}$ and $L(M_2)\subset k^{d_2}$, the tensor product $L(M_1)\otimes L(M_2)$ is a linear subspace of $k^{d_1d_2}$ whose corresponding matroid is a tensor product of $M_1$ and $M_2$. Given matrix representations of $L(M_1)$ and $L(M_2)$ their Kronecker product yields a matrix representation of $L(M_1)\otimes L(M_2)$.
    
    The following lemma is a result of LasVergnas, and is useful for understanding the structure of tensor products.
    \begin{lemma}\label{lem:tensorCharacterizations}\cite{LV81}*{Lemma 2.2}
    Let $M$ and $N$ be matroids and let $P$ be a quasi product of $M$ and $N$. The following are equivalent.
    \begin{enumerate}
        \item $\rk(P)=\rk(M)\rk(N),$
        \item for every basis $B_M$ of $M$ and every basis $B_N$ of $N$, $B_M\times B_N$ is a basis of $P$,
        \item for every pair of flats $F'\subset F$ of $M$ and $G'\subset G$ of $N$ 
        $$F'\times G \cup F\times G'$$
        is a flat of $P$,
        \item for every pair of sets $X'\subset X\subset E_M$ and $Y'\subset Y\subset E_N$
        $$\rk(X'\times Y\cup X\times Y')= \rk(X'\times Y)+\rk(X\times Y)- \rk(X'\times Y').$$
    \end{enumerate}
    \end{lemma}
    It is important to keep in mind that matroid tensor products may not be unique. In fact, even for a fixed field $k$ and two $k$-representable matroids $M_1$ and $M_2$, the tensor product obtained from linear subspace representations $L(M_1)$ and $L(M_2)$ may depend on the choice of representatives $L(M_1)$ and $L(M_2)$. 

    However, existence and uniqueness are guaranteed in some cases.

    \begin{lemma}\label{lem:tensorSums}
     Let $A,B,$ and $M$ be any matroids and let $N=A\oplus B$. If $P$ is a tensor product of $M$ and $N$ then $P= P_A\oplus P_B$ where $P_A$ is a tensor product of $M$ and $A$ and $P_B$ is a tensor product of $M$ and $B$.
    \end{lemma}
    \begin{proof}
        Consider the matroids 
        $$P_A = P|_{E(M)\times E(A)},$$
        and
        $$P_B = P|_{E(M)\times E(A)}.$$
        It is straightforward to check that $P_A$ and $P_B$ are quasi products of $M$ with $A$ and $B$ respectively. Since
        $$\rk(P)=\rank(M)\rank(A)+\rank(M)\rank(B))$$
        we may conclude that $\rk(P_A)=\rk(M)\rk(A)$ and $\rk(P_B)= \rk(M)\rk(B)$. Therefore $P_A$ and $P_B$ are tensor products. 
    \end{proof}

    \begin{lemma}\label{lem:tensorSimplification}
    Let $M$ and $N$ be any matroids. Then any quasi product of $M$ with $N$ depends only the simplification of $M$. 
    \end{lemma}
    \begin{proof}
        Let $P$ be a quasi product of $M$ with $N$. Suppose $M$ has a loop $e$. Then for any $x$ in $E(N)$, we have that $(e,x)$ is a loop of $P$. This is an immediate consequence of \coloref{eqn:tensorRight0}. In particular, adding or removing a loop $e$ from $M$ affects the quasi product $P$ by adding or removing the set of loops $\set{e}\times E(N)$ from $P$. 

        Now suppose that $S$ is a parallelism class of $M$. Then \colorefiso{eqn:tensorRight} implies that $S\times \set{x}$ is contained in a parallelism class $S'$ of $P$. In particular, replacing the parallelism class $S$ with a representative $e$ affects $P$ by replacing the subset $S\times \set{x}$ of the parallelism class $S'$ with $(e,x)$. 
    \end{proof}
\begin{proposition}\label{prop:uniqueTensors}
Let $M$ be any matroid. Then there is a unique tensor product between $M$ and $U_{n,n}$.
\end{proposition}
    \begin{proof}
            If $P$ is a tensor product of $M$ and $U_{n,n}$, then \coloref{lem:tensorSums} implies that $P = \oplus_{i=1}^n P_i$ where $P_i$ is a tensor product of $M$ and $U_{1,1}$. In particular $P_i\cong M$ and so $P=\oplus_{i=1}^n M$. We see that $\rk(P)=n\rk(M)=\rk(U_{n,n})\rk(M)$ and so $P$ is tensor product. This proves existence and uniqueness.
    \end{proof}

    \begin{corollary}\label{cor:uniqueSimplification}
        Let $M$ be any matroid and let $N$ be a matroid whose simplification is $U_{n,n}$. Then a tensor product between $M$ and $N$ exists and is unique.
    \end{corollary}
    \begin{proof}
    This follows from \coloref{lem:tensorSums}, \coloref{lem:tensorSimplification}, and \coloref{prop:uniqueTensors}
    \end{proof}
 
    Outside of the realizable case, little is known about the existence of matroid tensor products beyond \coloref{cor:uniqueSimplification}. It is known, however, that not all matroids have tensor products between them. 
    
    In his first work on matroid tensor products, Las Vergnas proved that the V\'amos matroid $V_8$ has no second tensor power. However, in \cite{LV81} this proof is simplified into two smaller results. First, \cite{LV81}*{Proposition 2.1} proves that no tensor product between $U_{2,3}$ and $V_8$ exists. Then \cite{LV81}*{Remark 2.3} provides that if a tensor power $P$ existed, it would have a minor isomorphic to a tensor product of $U_{2,3}$ and $V_8$, a contradiction. We conclude that no second tensor power of $V_8$ exists. In particular, the V\'amos matroid is a forbidden minor of the class of matroids with second tensor powers. We make this connection between tensor products and their minors explicit.

    We denote the collection of independent sets of a matroid $N$ by $\mathbf{\sI(N)}$.

    \begin{proposition}\label{prop:tensorMinorClosed}
        Let $M$ and $N$ be any matroids and let $P$ be a tensor product of $M$ and $N$. Let $S$ and $T$ be disjoint subsets of $E(M)$ and let $M'= M\sm S\contr T$ be a minor of $M$. Then the minor $P':=P\sm (S\times E(N))\contr (T\times E(N))$ is a tensor product of $M'$ with $N$.
    \end{proposition}
    \begin{proof}
        It suffices to prove the result for single-element deletions and contractions. Moreover, by \coloref{lem:tensorSimplification}, we may assume that $M$, $N$, and $M\sm e$ are simple matroids.

        First consider the minor $P':=P\sm(e\times E(N))$. We prove that $P'$ is a tensor product of $M\sm e$ and $N$.
        
        To see that $P'$ is a quasi product, first note that for $a\in E(M)\sm e$ we have $P'|_a\cong P|_a\cong N$, so \colorefiso{eqn:tensorRight} holds. Furthermore, for $b\in E(N)$ we have $P'|_b\cong (P|_b)\sm (e\times b)\cong M\sm e$, as desired. Now let $I_M$ be independent in $M\sm e$ and $I_N$ independent in $N$. Then $I_M$ is independent in $M$, and therefore $I_M\times I_N$ is independent in $P$. This implies $I_M\times I_N$ is independent in $P'$. We conclude by \coloref{lem:tensorCharacterizations} that $P'$ is a tensor product of $M\sm e$ and $N$.

        Now let $P':=P\contr (e\times E(N))$. We prove that $P'$ is a tensor product of $M\contr e$ and $N$. Consider the collection 
        $$a\times \sI(N):=\set{a\times I: I\in \sI(N)}.$$
        We prove that $\sI(P'|_a)= a\times \sI(N)$. First let $T$ be independent in $P'|_a$. Then there is an independent set $J'\in \sI(P')$ with $J'|_a=T$. Moreover, this implies there is an independent subset $J\in \sI(P)$ contained in $J'\cup (e\times E(N))$. In particular, $J|_a=T$ is an independent set in $P|_a\cong N$, so $T= a\times I$ for some $I\in \sI(N)$.

        For the reverse inclusion, suppose that $I$ is independent in $\sI(N)$ and let $B$ be a basis containing $I$. Then the set $(a\times I)\cup (e\times B)$ is contained in $\set{a,e}\times B$, which is independent in $P$. In particular, this implies $a\times I$ is independent in $P'$ and thus independent in $P'|_a$. This implies that \colorefiso{eqn:tensorRight} holds.

        Now let $b\in E(N)$. We proceed as before. First let $I_M$ be independent in $M\sm e$, and let $B_N$ be a basis of $N$ that contains $b$. Then $(I_M\cup e)\times B$ is independent in $P$ and contains the set $(I_M\times b)\cup (e\times B_N)$. In particular, $(I_M\times b)$ is independent in $P'$, and therefore in $P'|_b$.

        Now let $T$ be independent in $P'|_b$, say $T=S\times b$ for $S\subset E(M)\sm e$. Consider that since $T$ is independent in $P'|_b$, it is independent in $P'$, and therefore for any basis $B_N$ of $N$, the set $T\cup (e\times B_N)$ is independent in $P$. Suppose then that $B_N$ contains $b$. Then $\left(T\cup(e\times B_N)\right)_b = T\cup e$ is an independent set in $P|_b$. In particular $S\cup e$ is independent in $M$ and therefore $S$ is independent in $M\contr e$, as desired. We conclude $P'$ satisfies \colorefiso{eqn:tensorLeft}

        We now check that $P'$ is indeed a tensor product. Let $I_M$ be independent in $M\contr e$ and $I_N$ independent in $N$. We claim that $I_M\times I_N$ is independent in $P'$. To see this, let $B_N$ be a basis of $N$ containing $I_N$, then $(I_M\cup e)\times B_N$ is independent in $P$, or in particular, $(I_M\times I_N) \cup (e\times B_n)$ is independent in $P$. This implies $I_M\times I_N$ is independent in $P'$. \coloref{lem:tensorCharacterizations} now implies that $P'$ is a tensor product of $M\contr e$ and $N$, as desired.
    \end{proof}

    We now provide a generalization of the proof that the V\'amos matroid has no second tensor power, and we further extend this result to the infinite family of non-algebraic matroids introduced by Lindstr\"om \cite{Lind88}. The matroid $M(r)$ is a rank $r$ matroid defined as follows. Let $A:=\set{a_1,\dots, a_{r-2}}$, $B:=\set{b_1,\dots, b_{r-2}}$, $C:=\set{c_1,\dots, c_{r-2}}$, and $D:=\set{d_1,\dots, d_{r-2}}$ be disjoint sets of size $r-2$. We let the ground set of $M(r)$ be $E_r:= A\cup B\cup C\cup D$. We define the set of \textbf{cyclic flats} of $M(r)$ --- the flats that are unions of circuits --- as $\sH:= \set{A\cup B, A\cup C, A\cup D, B\cup C, B\cup D}$. The cyclic flats of $M(r)$, together with their rank, determine $M$ as a matroid \cite{Bry75}. In the case of $M(r)$, every cyclic flat in $\sH$ has rank $r-1$ and so they are hyperplanes. Observe that $M(4)$ is the V\'amos matroid.

    \begin{proposition}\label{prop:XandY}
        Let $[3]$ denote the ground set of $U_{2,3}$ and let $E:=A\cup B\cup C\cup D$ denote the ground set of $M(r)$. Let $P$ be any quasi product of $U_{2,3}$ and $M(r)$. Then either the set
        $$X:=([3]\times A)\cup (1\times B)\cup (2\times C)\cup (3\times D)$$
        or the set
        $$Y:= ([3]\times B)\cup (1\times A)\cup (2\times C)\cup (3\times D),$$
        is a spanning set.
    \end{proposition}
    \begin{proof}
    
        Let $P$ be a quasi product of  $U_{2,3}$ and $M(r)$. If one of $X$ or $Y$ is a spanning set, we're done. So suppose neither $X$ nor $Y$ is spanning. We claim that $X$ and $Y$ are both hyperplanes of $P$. To see this, we show that adding any element to either $X$ or $Y$ results in a spanning set. To make this more clear we refer to the ground set in terms of rows, indexed by the ground set $[3]$ of $U_{2,3}$, and columns indexed by the ground set $E$ of $M(r)$. We first consider $X$. Adding an element $e=(e_1,e_2)$ to $X$ leads to a subset of the form
            $$e_1\times (A\cup T\cup e_2),$$
    where $T\in \set{B, C, D}$. Now the closure of $X\cup e$ necessarily contains all of row $e_1$ as $A\cup T\cup e_2$ is a spanning set of $M(r)$. This now implies that $\ol{X\cup e}$ contains at least $2$ elements in each column of $[3]\times E$, which provides a spanning set of $U_{2,3}$ in each column. Thus the closure in every column is the entire column and so $\ol{X\cup e}=[3]\times E$. A similar argument holds for $Y$, where $Y\cup e$ contains the row $e_1\times (B\cup T\cup e_2)$ where $T\in{A,C,D}$. We conclude that $X$ and $Y$ are hyperplanes of $P$.

    Now consider that $Z:= X\cap Y$ is also a corank-1 set. This follows as as adding the element $e = (1,c_1)$ leads to a spanning set in row 1. Following the previous argument we see that $\ol{Z\cup (1,c_1)}$ contains at least 2 elements in the columns of $C$ and $D$, which is a spanning set in those column. So $\ol{(Z\cup(1,c_1))}$ contains the set $[3]\times (C\cup D)$. However, $C\cup D$ is a spanning set of $M(r)$, and so every row of $\ol{(Z\cup(1,c_1))}$ contains a spanning set. Therefore $\ol{(Z\cup(1,c_1))}= [3]\times E$. This is a contradiction as the intersection of two distinct hyperplanes, namely $X$ and $Y$, must have corank at least $2$, whereas $Z$ has corank $1$. Thus either $X$ or $Y$ must be a spanning set.
    \end{proof}
    
    \begin{proposition}\label{prop:VxU23}
    Any quasi product of $U_{2,3}$ and $M(r)$ has rank at most $2r-1$. In particular, no tensor product between $U_{2,3}$ and $M(r)$ exists.
    \end{proposition}
    \begin{proof}
        Let $[3]$ denote the ground set of $U_{2,3}$ and let $E:=A\cup B\cup C\cup D$ denote the ground set of $M(r)$. Let $P$ be a quasi product of $U_{2,3}$ and $M(r)$. Suppose $X$ and $Y$ are as in \coloref{prop:XandY}, and without loss of generatlity let $X$ be a spanning set among $X$ and $Y$. We prove $X$ has rank at most $2r-1$. To see this, consider that $X$ is necessarily spanned by $([2]\times A)\cup \set{(1,b), (2,c), (3,d)}$, which contains $2r-1$ elements, thus $\rk(P)=\rk(X)\leq 2r-1$.
    \end{proof}
    
    As a consequence of \coloref{prop:VxU23}, the matroids $M(r)$ have no tensor powers, and we further note the following small corollary
    
    \begin{corollary}\label{cor:VxU23}
    A tensor product between $M(r)$ and $N$ exists if and only if the simplification of $N$ is a free matroid.
    \end{corollary}
    \begin{proof}
    If the simplification of $N$ is not free, then $N$ contains a circuit of size at least 3, and therefore $N$ contains $U_{2,3}$ as a minor. By \coloref{prop:tensorMinorClosed} and \coloref{prop:VxU23}, no tensor product between $M(r)$ and $N$ can exist. Otherwise, the simplification of $N$ is free and \coloref{cor:uniqueSimplification} implies that $M(r)$ and $N$ have a unique tensor product.
    \end{proof}
    
    In addition to showing that each $M(r)$ has no tensor product with $U_{2,3}$ we draw the further conclusion that there is an infinite family of forbidden minors to the class of matroids with second tensor power, which is minor closed by \coloref{prop:tensorMinorClosed}.
    \begin{corollary}\label{cor:tensorForbiddenMinors}
        The class of matroids with second tensor power has infinitely many forbidden minors.
    \end{corollary}
    \begin{proof}
    The matroids $M(r)$ satisfy the following properties \cite{Lind88}
    \begin{enumerate}
        \item $M(r)$ is a non-algebraic matroid,
        \item for any $e\in E_r$, the contraction $M(r) \contr e$ is realizable,
        \item if $r_1< r_2$, and two restrictions satisfy $M(r_1)|_{T_1}\cong M(r_2)|_{T_2}$, then $M(r_1)|_{T_1}\cong U_{r,|T_1|}$ for either $r= r_1$ or $r=|T_1|$.
    \end{enumerate}
    We deduce from these observations that, for $r_1\neq r_2$, any minor common between $M(r_1)$ and $M(r_2)$ is a realizable matroid. In particular, any common minor has second tensor powers. We conclude that each $M(r)$ has a minor, which is not a minor to any other $M(r')$ that is a forbidden minor of the class of matroids with second tensor powers.
    \end{proof}

    While the connection between tensor products and tropical geometry is not as immediate as will be seen with symmetric powers, we are nonetheless able to use tensor products in order to prove certain symmetric powers do not exist, as done in \cite{DR19}.

\subsection{Matroid Symmetric Powers}
We now turn our focus onto symmetric powers of matroids. Recall that, for any set $E$, the $d\th$ symmetric power of $E$ is the set of unordered $d$-tuples of elements of $E$, which we denote $\Sym_d(E)$. If we fix an element $\alpha \in \Sym_i(E)$, there is an injective ``multiplication" map
    \begin{align*}
        *_\alpha: \Sym_d(E)&\to \Sym_{d+i}(E)\\
        \beta&\mapsto (\alpha,\beta).
    \end{align*}
Since the coordinates of elements of $\Sym_{d+i}$ are unordered, we will write $(\alpha,\beta)$ as $\alpha\beta$, taking after commutative multiplication notation. 

We will often consider the subset of elements of $\Sym_{d+i}(E)$ that are divisible by some $\alpha\in \Sym_{i}(E)$, which we denote as

$$\Sym_{d+i}(E)|_{\alpha} := \set{e\in \Sym_{d+i}(E)\mid e= \alpha x}.$$

Given a valuated matroid $\sV$ on $\Sym_{d+i}(E)$ we further extend this notation and write $\sV|_{\alpha}$ instead of $\sV|_{\Sym_{d+1}(E)|_{\alpha}},$ where the restriction of $\sV$ to the set $T\subset [n]$ is given by the subsemimodule 
$$S|_T:=\sV\cap \set{x\in \T^n \mid x_i=\infty \text{ for all } i\notin T}.$$

\begin{definition}\label{def:symprod}
Fixing a matroid $\sV$ on a ground set $E$, a $\mathbf{d}^{\text{\textbf{th}}}$ \textbf{symmetric quasi power} of $\sV$ is a sequence of matroids $(\sV_1,\dots,\sV_d)$ with $\sV_1 = \sV$ and for all $1\leq i \leq d$
\begin{itemize}
    \item $E(\sV_i) = \Sym_i(E),$
    \item for all $\alpha\in \Sym_{d-i}(E)$ for which $\alpha$ is not divisible by a loop the multiplication map $*_\alpha$ given by $y\mapsto \alpha x$ induces the isomorphism
    \begin{equation}\label{eqn:symprodEmbedding} 
        \sV_i\cong \sV_d|_{\alpha}
    \end{equation}
    otherwise
    \begin{equation}\label{eqn:symprodLoops}
        \rk(\sV_d|_{\alpha})=0.
    \end{equation}
    
\end{itemize} 
Moreover, if $(\sV_1,\dots,\sV_d)$ is a symmetric quasi power of $\sV$ with

\begin{equation}\label{eqn:symprodRank}
\rk(\sV_d)= \binom{\rk(\sV)+d-1}{d}
\end{equation}

then we say that $(\sV_1,\dots,\sV_d)$ is a $\mathbf{d}^{\text{\textbf{th}}}$ \textbf{symmetric power} of $\sV$.
\end{definition}

The most familiar and pertinent example of symmetric powers arises from commutative polynomial rings. Given a polynomial ring $k[x_1,\dots,x_n]$ over a field (or semifield) $k$, the degree $d$ monomials in $k[x_1,\dots,x_n]$ exactly correspond to the set $\Sym_d(\set{x_1,\dots,x_n}),$ which we denote by $\Mon_d(n)$ --- the degree $d$ monomials in $n$ variables. The multiplication map $*_\alpha$ is then given by multiplication by the monomial $x^\alpha$ in $k[x_1,\dots,x_n]$.

In the setting of tropical polynomial semirings and symmetric powers, matroids can be constructed as collections of linear tropical polynomials, and their $d\th$ symmetric powers as collections of degree-$d$ tropical polynomials.

\begin{remark}
In \coloref{def:symprod} we define a $d\th$ symmetric power of a fixed valuated matroid $\sV$ to be a sequence $(\sV_1,\dots,\sV_d)$. However, \colorefiso{eqn:symprodEmbedding} provides that $\sV_d$ uniquely determines the entire sequence $(\sV_1,\dots,\sV_d)$, so nothing is lost when referring to $\sV_d$ alone. However, $\sV$ may have multiple $d\th$ symmetric powers, as we will see in \coloref{ex:matrixSymprodExample}, and when constructing a $d\th$ symmetric power it will be useful to explicitly record the entire sequence $(\sV_1,\dots,\sV_d)$. We will further discuss the space of all symmetric powers of a fixed valuated matroid $\sV$ in \coloref{subsec:SpacesofSymPow}.
\end{remark}

\begin{remark}\label{rem:truncation}
If $(\sV_1,\dots,\sV_d)$ is a $d\th$ symmetric power of $\sV$, then for every $1\leq i\leq d$ we have that the truncation $(\sV_1,\dots,\sV_i)$ is an $i\th$ symmetric power of $\sV$.
\end{remark}

\begin{remark}\label{rem:loopless}
\colorefiso{eqn:symprodEmbedding} implies that if $\sV$ is loop-free, then so is every symmetric quasi power of $\sV$. Moreover, \coloref{eqn:symprodLoops} implies that adding or removing loops to $\sV$ to obtain a matroid $\sV'$ uniquely determines how to add or remove loops from any given $d\th$ symmetric quasi power $(\sV_1,\dots,\sV_d)$ to obtain a $d\th$ symmetric quasi power of $\sV'$. In particular, we may assume without loss of generality that $\sV$ is loop-free and all subsequent results extend to matroids with loops.
\end{remark}

\begin{remark}\label{rem:isoMinors}
If $\mu$ is a valuated matroid on $n$ elements and $\mu_d$ is a $d\th$ symmetric quasi power, then for any $\bfu_1, \bfu_2\in \Sym_{i}(E)$ and any $\bfv\in \Sym_{d-i}$, \colorefiso{eqn:symprodEmbedding} implies that 
\begin{equation}\label{eqn:isoMinors}
    (\mu_d|_{\bfu_1})\sm \bfu_1\bfv\cong (\mu_d|_{\bfu_2})\sm \bfu_2\bfv
\end{equation}
This follows as $\mu_d|_{\bfu_1}\cong \mu_{d-i}$ via the isomorphism $\bfu_1 \bfx\mapsto \bfx$. Similarly, $\mu_d|_{\bfu_2}\cong \mu_{d-i}$ via $\bfx\mapsto \bfu_2 \bfx$. \colorefiso{eqn:isoMinors} follows.
\end{remark}

It is also important to note that we refer to \textit{a} symmetric power of a matroid rather than \textit{the} symmetric power, as it may not be unique, not even for realizable matroids.

The following is stated in \cite{MGlueing}, but we restate it here and include a brief proof for intuition's sake.

\begin{proposition}\cite{MGlueing}\label{prop:masonFullRank}
    For a matroid $M$ of rank $r$ the rank of any symmetric quasi power $M_d$ is at most $\binom{r+d-1}{d}$ and at least $r$.
\end{proposition}

\begin{proof}
For any symmetric quasi power $M_d$ of $M$ and any basis $B$ of $M$, \colorefiso{eqn:symprodEmbedding} implies that the set $\Sym_d(B)$ is a spanning set of $M_d$. Thus the rank of $M_d$ can never be greater than $|\Sym_d(B)|=\binom{r+d-1}{d}$.

\colorefiso{eqn:symprodEmbedding} further implies that for any $e\in E(M)$ and any basis $B\in \sB(M)$, that $eB$ is independent in $M_d$. In particular, $\rk(M_d)\geq r$.
\end{proof}
 In light of \coloref{prop:masonFullRank}, we see that symmetric powers are exactly those symmetric quasi-powers that achieve the upper bound on the rank. Moreover given any field $k$, every $k$-representable matroid $M$ has a $k$-representable $d\th$ symmetric power that can be directly constructed. Given any $k$-matrix $L$ representing $M$ we can construct the $d\th$ \textbf{Macaulay matrix} of $L$, which will be a representation of a $d\th$ symmetric power of $M$. We outline the construction of the second Macaulay matrix here. Let $r\leq n$ and let $L$ be a rank-$r$ $r\times n$ matrix over a field $k$. The second Macaulay matrix of $L$, which we will denote $\Sym_2(L)$, is an $\binom{r+1}{2}\times \binom{n+1}{2}$ matrix whose rows are indexed by $\Sym_2([r])$ and whose columns are indexed by $\Sym_2([n])$. Fixing the standard basis $e_1,\dots,e_r$ of $k^r$, we can write each column $C_j$ of $L$ as a formal sum 
$$C_j = \sum_{i=1}^r L_{i,j}e_i.$$
We can now consider $C_j$ as a linear form in the polynomial ring $k[u_1,\dots,u_r]$ under the inclusion $\phi:k^n\mapsto k[u_1,\dots,u_r]$, $e_i\mapsto u_i$.

The column $ij$ of $\Sym_2(L)$ is computed by taking the product $\phi(C_i)\phi(C_j)\in k[u_1,\dots,u_r]_2$, which is then a formal sum in the basis $\Mon_2(r)$, i.e in basis vectors indexed by $\Sym_2([r])$. In particular we express 
$$\phi(C_i)\phi(C_j)=\sum_{x,y\in[r]} c_{xy,ij}u_xu_y.$$

We now define the entries of the second Macaulay matrix $\Sym_2(L)$ to be, for $x,y\in [r]$ and $i,j\in[n]$
$$\Sym_2(L)_{xy,ij} = c_{xy,ij},$$
that is, the $xy$ coefficient of the product of the polynomials associated to columns $C_i$ and $C_j$. This definition extends to defining $\Sym_{d}(L)$ by considering products in $k[x_1,\dots,x_n]_d.$

As indicated in our notation above, the $d\th$ Macaulay matrix of $L$ is indeed a matrix that represents a $d\th$ symmetric power of the matroid $M$ that $L$ represents. We will demonstrate this for the second Macaulay matrix. 
\begin{example}\label{ex:matrixSymprodExample}
Consider the matrix representing $U_{2,4}$ over any field other than $k_2$ with $\alpha \neq 0,1$
$$M:=\begin{bmatrix}
1 & 0 & 1 & 1\\
0 & 1 & 1 & \alpha
\end{bmatrix}.$$
The second Macaulay matrix of $M$ is a matrix whose column space represents a second symmetric power of $M$. It's columns are indexed by pairs of columns of $M$ and its rows are indexed by pairs of column entries of some fixed basis of $M$, in this case we use the first two columns of $M$.
$$\Sym_2(M):=\kbordermatrix{
        & c_{11}& c_{12}& c_{13}& c_{14} & c_{22}& c_{23}& c_{24} & c_{33}& c_{34}  & c_{44}  \\
r_{11}  &1      & 0     & 1     & 1      & 0     & 0     & 0      & 1     & 1       & 1       \\
r_{12}  &0      & 1     & 1     & \alpha & 0     & 1     & 1      & 2     & \alpha+1& 2\alpha \\
r_{22}  &0      & 0     & 0     & 0      & 1     & 1     & \alpha & 1     & \alpha       & \alpha^2}.$$
Note that the symmetric power of the fixed basis $\set{c_1,c_2}$ is  $\set{c_{11},c_{12},c_{22}}$, which is a basis of $\Sym_2(M),$ so it is indeed rank $3$ as desired.

Consider the minor $B$ consisting of the following columns of $\Sym_2(M)$
$$B:=\kbordermatrix{
        &c_{11} & c_{22}    & c_{34}    \\
r_{11}  &1      & 0         & 0         \\
r_{12}  &0      & 0         & \alpha +1    \\
r_{22}  &0      & 1         & \alpha  }.$$
We can see that $B$ has determinant $-(\alpha+1)$. In particular if the characteristic is not $2$ then we may choose $\alpha= -1$ and $B$ will not be a basis. Otherwise, if we can choose $\alpha\neq 1,0,-1$ then $B$ will be a basis. Similarly, the matrix
$$C:=\kbordermatrix{
        &c_{11} & c_{22}    & c_{44}    \\
r_{11}  &1      & 0         & 1         \\
r_{12}  &0      & 0         & 2\alpha    \\
r_{22}  &0      & 1         & \alpha^2  }.$$
has determinant $-2\alpha$, which is 0 if and only if $k$ has characteristic 2. In particular is never a basis when represented in characteristic $2$ and it is always a basis otherwise.
\end{example}

\begin{remark}
\coloref{ex:matrixSymprodExample} shows that the underlying matroid of the second Macaulay matrix depends both on the field over which we take the representation $L$, as well as the representation $L$ itself over any fixed field. However, if $L$ is a representation of $M$ over the field $F$, then the second Macaualay matrix is preserved under action of $\GL_r(F)$ by left-multiplication in the following sense. If $U\in \GL_r(F)$ we have
    $$\Sym_2(UL) = \Sym_2(U)\Sym_2(L).$$
\coloref{ex:matrixSymprodExample} demonstrates that the analogue for right-hand multiplication by $V\in \GL_n(F)$ does not hold.
\end{remark}

We will prove in \coloref{cor:minorClosed} that the class of matroids with $d\th$ symmetric power is minor-closed. In light of \coloref{cor:tensorForbiddenMinors}, we can further show that the class of matroids with $d\th$ symmetric power has infinitely many forbidden minors
\begin{theorem}\label{thm:sympowForbiddenMinors}
    The class of matroids with $d\th$ symmetric powers has infinitely many forbidden minors.
\end{theorem}
\begin{proof}
    By \coloref{rem:truncation} it is sufficient to construct a family of forbidden minors for the class of matroids with second symmetric power. 

    Draisma and Rincon prove that if $M$ and $N$ have no second tensor power, then $M\oplus N$ has no second symmetric power \cite{DR19}. By \coloref{prop:VxU23}, for every $r\geq 4$, the matroids $M(r)\oplus U_{2,3}$ have no second symmetric power.

    We further claim that for integers $r_1>r_2\geq 4$, the matroids $M_1:=M(r_1)\oplus U_{2,3}$ and $M_2:=M(r_2)\oplus U_{2,3}$ have no common minors that are not realizable.
    
    We recall that, for all $r$, any contraction $M(r)\contr e$ is realizable, and all deletions common to $M(r_1)$ and $M(r_2)$ are realizable matroids \cite{Lind88}. In particular, a non-realizable minor of $M_1$ has circuits of size $r_1$, $r_1+1$, and at most one of size $3$ or less. Similarly, a non-realizable minor of $M_2$ has circuits of size $r_2$, $r_2+1$ and at most one of size $3$. Since $r_1>r_2>3$, these circuit sets cannot be equal and we conclude that there can be no common non-realizable minors between $M_1$ and $M_2$. The result follows.
\end{proof}

\subsection{Symmetric Powers by Flats}\label{subsec:sympowByFlats}
In \coloref{def:symprod} we decided on an axiomatization for matroid symmetric powers that is potentially stronger than that which is proposed by Mason in \cite{MGlueing}. Originally motivated from Lov\'asz's early work with symmetric powers of realizable matroids \cite{Lovasz77}, Mason proposed the following condition for a symmetric quasi power, as defined in  \coloref{def:symprod}, to be a symmetric power.
\begin{equation}\label{axiom:flatAxiom} 
\text{For every flat } F \text{ of } M \text{ the set } F\cdot \Sym_{d-1}(E) \text{ is a flat of } M_d.
\end{equation}

In \coloref{def:symprod} we instead required that \coloref{eqn:symprodRank} be satisfied.

\begin{proposition}\label{prop:SymprodImpliesFlats}
    If $M$ is a valuated matroid and $(M_1,\dots, M_d)$ is a $d\th$ symmetric power in the sense of \coloref{def:symprod}, then $M_d$ satisfies \colorefaxiom{axiom:flatAxiom}.
\end{proposition}

We will prove \coloref{prop:SymprodImpliesFlats} in \coloref{subsec:sympowCharacterizations}, using our new geometric techniques to avoid an in-depth matroid theoretical argument. While we are able to prove that \coloref{def:symprod} implies \colorefaxiom{axiom:flatAxiom}, the reverse is not at all obvious. Despite substantial effort, we were not able to make progress on this problem and have left it as an open problem left by Lov\'asz and Mason.

\begin{conjecture}(Lovasz '77, Mason '81)\label{conj:flatcondition}
If $(M_1,\dots,M_d)$ is a symmetric quasi power of $M$ and for every $F\in \sF(M)$ we have $F\cdot \Sym_{d-1}(E)\in \sF(M_d)$, then $M_d$ has rank $\binom{\rk(M)+d-1}{d}$.
\end{conjecture}

In order to prove the conjecture, it is sufficient to prove the following auxiliary result.

\begin{conjecture}\label{conj:flatcondition2}
    If $M$ is a matroid on the ground set $E$, $(M_1,\dots,M_d)$ is a symmetric quasi power of $M$, and for every $F\in \sF(M)$ we have $F\cdot \Sym_{d-1}(E)\in \sF(M_d)$, then there is an $e\in E$ such that the contraction
    $$M_d\contr (e\cdot \Sym_{d-1}(E))$$
    is a symmetric quasi power of $M\contr e$.
\end{conjecture}

\coloref{conj:flatcondition2} is nontrivial in general, as it is possible for $(M_1,\dots,M_d)$ to be a symetric quasi power, that is not full rank, and for the contraction $M_d\contr (e\cdot \Sym_{d-1}(E))$ to not be a symmetric quasi power for some $M\contr e$. However, it is not clear if this is necessarily possible that this could be the case for all $e$. It is not clear if \coloref{conj:flatcondition} is true, not even for second symmetric powers,  without additional substantial hypotheses.

\subsection{Further characterizations of Matroid Symmetric Powers}\label{subsec:sympowCharacterizations}

In this section we will discuss characterisations of symmetric powers, similar to how \coloref{lem:tensorCharacterizations} treats tensor products. In particular, we will provide a geometric characterization, alongside a characterisation by flats that requires a strictly stronger condition than that of \coloref{conj:flatcondition}.

In order to establish a geometric interpretation of matroid symmetric powers, we will need to define the \textit{variety} of a symmetric power.

\begin{definition}\label{def:matroidVariety}
Let $M$ be a matroid on ground set $[n]\cong\set{x_1,\dots,x_n}$, and let $M_d$ be a symmetric quasi power of $M$ on ground set $\Sym_{d}([n])\cong \Mon_d(n)$. For each circuit $C$ of $M_d$ we define the \textit{circuit polynomial}
    $$p_C := \bigoplus_{x\in C} x\in \T[x_1,\dots,x_n]_d.$$

We can now define the (tropical) \textbf{variety} of $M_d$ as
$$V(M_d) : = \bigcap_{C\in \sC(M_d)} V\left(p_C\right)\subset \T^n.$$
\end{definition}

Given a matroid $M$ on ground set $[n]$, its variety $V(M)\subset \T^n$ is equal to its \textbf{Bergman Fan}. For a symmetric power $M_d$ its variety and Bergman fan are distinct geometric objects. The Bergman fan treats each element of $\Sym_d(E)$ as an independent variable whereas the variety takes into account the algebraic dependencies between the variables. In particular, $B(M_d)\subset\T^{|\Sym_d(E)|}$ and $V(M_d)\subset \T^n$. We further remark that it is always the case that $V(M_d)\subset B(M)$ and that $V(M_d)$ maps to $B(M_d)$ through the Veronese embedding.

\begin{proposition}\label{prop:symprodCharacterisations}
Let $M_d$ be a $d\th$ symmetric quasi power of $M$. The following are equivalent for $M_d$
\begin{enumerate}[label=(\roman*),]
    \item\label{prop:symprodCharacterisations.rk} $M_d$ has rank $\binom{\rk(M)+d-1}{d}$,
    \item\label{prop:symprodCharacterisations.bases} for every basis $B$ of $M$, $\Sym_d(B)$ is a basis of $M_d$,
    \item\label{prop:symprodCharacterisations.bf} the variety $V(M_d)$ is supported on $B(M)$,
    \item\label{prop:symprodCharacterisations.flats} for every $\omega\in B(M)$ the set
    $$\set{x\in \Sym_d(E)\mid x(\omega)>0}$$
    is a flat of $M_d$.
\end{enumerate}
\end{proposition}
The condition imposed by \coloritemref{prop:symprodCharacterisations}{flats} is stronger than \colorefaxiom{axiom:flatAxiom}, which can be phrased in the following way.
\begin{quote}
    For every $\omega\in B(M)\cap \R^n_{\geq 0}$. The set
    $$\set{x\in \Sym_d(E)\mid x(\omega)>0}$$
    is a flat of $M_d$.
\end{quote}
Recall that the lineality space of $B(M)$ contains $\R \mathbbm{1}$. In particular, the positive part of the Bergman fan determines the entire fan. The discrepancy between \coloref{conj:flatcondition} and \coloritemref{prop:symprodCharacterisations}{flats} is a matter of whether or not we need to consider the flats corresponding to every point in the Bergman fan, or just the positive part --- the part that determines the entire fan.

\begin{proof} The equivalence between \coloritemref{prop:symprodCharacterisations}{rk}
and \coloritemref{prop:symprodCharacterisations}{bases} follows from \coloref{prop:masonFullRank}. We will prove the equivalence between \coloritemref{prop:symprodCharacterisations}{rk} and \coloritemref{prop:symprodCharacterisations}{bf} in \coloref{cor:bfCharacterisation}. We now prove the equivalence between \coloritemref{prop:symprodCharacterisations}{bf} and \coloritemref{prop:symprodCharacterisations}{flats}.

Suppose that $V(M_d)\neq B(M)$. Then $V(M_d)\subsetneq B(M)$ and so we let $\omega\in B(M)\sm V(M_d)$. Since $\omega\notin V(M_d)$ there is a circuit $C$ of $M_d$ that attains its minimum exactly once at $\omega$, say at the monomial $x$. Thus $x(\omega)$ is the unique minimum value of $C(\omega)$. Since the lineality space of $B(M)$ contains $\R\mathbbm{1}$ we may translate $\omega$ to $\omega' = \omega- \frac{x(\omega)}{d} \mathbbm{1}$ so that $x(\omega')=0$ is the unique minimum value of $C$ at $\omega'$. Now consider that this implies the set
$$S:=\set{y\in \Sym_d(E)\mid y(\omega')>0}$$
contains all of $C$ except for $x$. Thus $S$ is not a flat, as desired.

Now suppose that the set 
$$S = \set{y\in \Sym_d(E)\mid y(\omega)>0}$$
is a not a flat for some $\omega\in B(M)$. We claim that $V(M_d)\neq B(M)$. Since $S$ is not a flat there is a circuit $C$ such that $|C\sm S|=\set{x}$ --- a singleton. In particular, $x(\omega)$ is the unique minimum of $C$ at $\omega$ and so $\omega\notin V(M_d)$ and $V(M_d)\neq B(M)$, as desired.
\end{proof}

\begin{proof}[Proof of \coloref{prop:SymprodImpliesFlats}]
If $M$ is a matroid and $(M_1,\dots,M_d)$ is a symmetric power of $M$ as in \coloref{def:symprod}, then by \coloref{prop:symprodCharacterisations}, for every $\omega\in B(M)$ we have that

$$S_\omega:= \set{x\in \Sym_d(E): x(\omega)>0}.$$
Let $F$ be a flat of $M$. The set $F\cdot \Sym_{d-1}(E)$ is the set of elements of $\Sym_d(E)$ divisible by some element of $F$. Now let $\omega= e_F$. A monomial $x$ satisfies $x(\omega)>0$ if and only if $x$ is divisible by an element of $F$. In particular $F\cdot\Sym_{d-1}(E)=S_\omega$ and is therefore a flat of $M_d$, as desired.
\end{proof}

\subsection{Open Questions}\label{subsec:openQuestions}

We conclude this section by posing a few open questions that we were not successful in pursing in this paper. In the case of linear spaces there is a well-established connection between tensor and symmetric powers of linear spaces via the quotient of the action of the symmetric groups, we pose this as a question in the context of matroid powers

\begin{question}
Is there a relationship, e.g via matroid quotients, between matroid tensor powers and matroid symmetric powers? Are there suitable hypothesis on a matroid tensor power to guarantee the existence of a matroid symmetric power?
\end{question}

 We might further refine the previous question and ask if symmetric powers are, in some sense,  easier to construct than tensor powers.
\begin{question}\label{question:symprodsWithoutTensor}
Can a matroid have a $d \th$ symmetric power and no $d\th$ tensor power, or vice-versa?
\end{question}

A prime candidate for answering this question is to determine if the V\'amos matroid $V_8$ has a second symmetric power. We imagine that it does not, but if it does in fact have one, it would resolve this question and also prove that the class of matroids with second symmetric powers is not closed under duality. As a quick proof of this, if $V_8$ has a second symmetric power $P$, then so would $V_8\oplus U_{1,3}$ via $P\oplus (V_8\otimes U_{1,3})\oplus (U_{1,3})_2$, where $V_8\otimes U_{1,3}\cong V_8$. However, since $V_8$ is self-dual, the dual of $V_8\oplus U_{1,3}$ is $V_8\oplus U_{2,3}$, whose second symmetric power does not exist \cite{DR19}*{Theorem 5.2}. We refine this into its own question.

\begin{question}\label{question:vamosSym2}
Does the V\'amos matroid have a second symmetric power?
\end{question}
%%%%%%%%%%%%%%%%%%%%%%%%%%%%%%%%%%%%%%%%%%%%%%%%%%%%

%%%%%%%%%%%%%%%%%%%%%%%%%%%%%%%%%%%%%%%%%%%%%%%%%%%%

%%%%%%%%%%%%%%%%%%%%%%%%%%%%%%%%%%%%%%%%%%%%%%%%%%%%
\section{Linear Tropical Ideals}\label{sec:LinTropIdeals}

   This section will be focused on what we will call \textit{linear tropical ideals}, the tropical analogue of linear ideals. In the realizable case, linear ideals are perhaps the most simple and structured ideals of positive dimension, and any linear space can be realized as the variety of some linear ideal. However, when working with tropical ideals, valuated matroids, and their associated tropical linear spaces, it is already known that not all tropical linear spaces are tropically realizable varieties \cite{DR19}. Despite this complication in the theory, tropical linear spaces remain a reasonably well-behaved and well-studied class of tropical varieties, see \cite{Speyer08}, \cite{FR15}, \cite{GG18} for a few examples. In the remainder of this paper we will establish further combinatorial and geometric connections.
   
   A homogeneous ideal $I\subset k[x_1,\dots,x_n]$ is called linear if it is generated by linear forms or, equivalently, it is the ideal of polynomials vanishing on a linear space. While tropical ideals are generally not finitely generated, they do have a finite \textbf{tropical basis} \cite{MR18}*{Theorem 5.9}, which is a set of polynomials in the ideal that cuts out the variety. This inspires the following definition.
   
    \begin{definition}    
    A homogeneous tropical ideal $I\subset \T[x_1,\dots, x_n]$ is a \textbf{linear tropical ideal} if its linear forms form a tropical basis, that is
    $$V(I)= V(I_1).$$
    \end{definition}
    The variety of a linear tropical ideal is necessarily a tropical linear space balanced with constant weights $1$.

    We say that an ideal $I\subset \T[x_1,\dots,x_n]$ is \textbf{saturated} with respect to a polynomial $f$ if $fp\in I$ implies $p\in I$. When $f=\prod_{i=1}^n x_i$ we will simply say that $I$ is saturated.
    
    \begin{remark}\label{rem:linSaturated}
    Given a linear tropical ideal $I\subset \T[x_1,\dots, x_n]$, we may express $I$ as a sequence of valuated matroids $(I_1,I_2,\dots,I_d,\dots)$. We will see that $I$ is saturated and that the truncation $(I_1,\dots,I_d)$ is a $d\th$ symmetric quasi-power of the matroid $I_1$.
    \end{remark}

\subsection{Connectivity}\label{subsec:connectivity}
 In order to prove the main results of this section, we will need to show that a $d$-dimensional tropical linear space has no proper tropical subvarieties of dimension $d$. To show this, we will prove a stronger result, which extends a valuable result in matroid theory to all valuated matroids. Specifically, we will show that the space of top-dimensional Minkowski weights of a tropical linear space is one-dimensional. We refer the reader to \cite{MS15}[Section 3.3] for relevant definitions and background for weighted polyhedral complexes, balancing, and the $\Star$ operation for a polyhedral complex.

If $\Sigma$ is a polyhedral complex in $\R^n$, we write $|\Sigma|$ to refer to the \textbf{support} of $\Sigma$, which is the set of points in $\Sigma$ without the polyhedral structure. Recall that the facet-ridge graph of $\Sigma$ is the graph whose vertices are given by the facets of $\Sigma$ and where there are edges between two vertices $\sigma_1$ and $\sigma_2$ if $\sigma_1$ and $\sigma_2$ intersect along a ridge of $\Sigma$. We say that $\Sigma$  is connected through codimension one if its facet-ridge graph $G(\Sigma)$ is a connected graph. We additionally say that $\Sigma$ is locally connected through codimension 1 if $G(\Star_w(\Sigma))$ is connected for every $w\in |\Sigma|$. 

\begin{proposition}\label{prop:connectedCodim1}
Let $\Sigma$ be a connected pure polyhedral complex that is locally connected through codimension one. Then $\Sigma$ is connected through codimension one.
\end{proposition}
\begin{proof}
    Let $\Sigma\subset \R^n$ be a connected, pure, polyhedral complex that is locally connected through codimension one. Further suppose towards a contradiction that $\Sigma$ is not connected through codimension one. Then the facet-ridge graph $G(\Sigma)$ has at least two connected components $C_1, C_2,\dots$. Since the vertices of $C_i$ are cones of $\Sigma$ we define the support of $C_i$ as
    $$|C_i|:= \bigcup_{\sigma\in V(C_i)} \sigma\subset |\Sigma|.$$
    Since $\Sigma$ is connected, there must exist $C_i$ and $C_j$ such that $|C_i|\cap |C_j|\neq \es$. Without loss of generality we suppose, suppose that $|C_1|\cap |C_2|\neq \es$, and let $w\in |C_1|\cap |C_2|$. Since $w\in |C_1|$ and $|C_2|$ there is are facets  $\sigma_1\in V(C_1)$ and $\sigma_2\in V(C_2)$ such that $w\in \sigma_1$ and $\sigma_2$. Now, since $\Sigma$ is locally connected through codimension one, $G(\Star_w(\Sigma))$ is a connected graph. Moreover, by construction, $G(\Star_w(\Sigma))$ is a subgraph of $G(\Sigma)$. Since $\sigma_1$ and $\sigma_2$ are vertices of $\Star_w(\Sigma)$, there is a path connecting $\sigma_1$ and $\sigma_2$ in $G(\Star_w(\Sigma))$, however this now implies there is a path between $\sigma_1$ and $\sigma_2$ in $G(\Sigma)$. Therefore $C_1$ and $C_2$ are not distinct connected components, a contradiction. We conclude that $G(\Sigma)$ is necessarily a connected graph and therefore $\Sigma$ is connected through codimension one.
\end{proof}

The conditions of \coloref{prop:connectedCodim1} are met by a large class of objects studied in tropical Geometry called tropical manifolds, which are meant to be a tropical analogue to classical manifolds. A $d$-dimensional \textbf{tropical manifold} is a polyhedral complex such that, for all $w\in |\Sigma|$, $\Star_w(\Sigma)$ is linearly isomorphic to the Bergman fan of some matroid that has dimension $d$. Of particular value to this paper is that every tropical linear space is a tropical manifold.

\begin{corollary}\label{cor:connectedCodim1}
    Any connected tropical manifold is connected through codimension one.
\end{corollary}
\begin{proof}
    The Bergman fan of any matroid is connected in codimension one. Since a connected tropical manifold $M$ is, by construction, locally the Bergman fan of a matroid, it is locally pure and locally connected in codimension one. This implies that $M$ is both pure and, by \coloref{prop:connectedCodim1}, connected in codimension 1.
\end{proof}

\begin{definition}
    Let $\Sigma$ be a polyhedral complex. A $k\th$ Minkowski-weight of $\Sigma$ is an rational-valued function from the $k$-skeleton of $\Sigma$, denoted $\Sigma_k$, to the rationals
    \begin{align*}
        m:\Sigma_k \to \Q
    \end{align*}
    such that $\Sigma_k$ is a balanced polyhedral complex with weights given by $m$. 
\end{definition}

\begin{corollary}\label{cor:minkWeights}
The space of top-dimensional Minkowski-weights of a tropical manifold is a $1$-dimensional $\Q$-vector space and is generated by the constant weight function $w(\sigma)=1$. 
\end{corollary}
\begin{proof}
    Let $\Sigma$ be a $d$-dimensional tropical manifold. Let $w$ be a $d\th$ Minkowski weight of $\Sigma$ and suppose without loss of generality that $w(\sigma)=1$ for some facet $\sigma$. We claim that $w(\sigma')=1$ for all other facets of $\Sigma$. Notice that any facet $\sigma'$ adjacent to $\sigma$ in the facet-ridge graph $G(\Sigma)$ must have weight $1$. This follows as there exists a $w\in \sigma\cap\sigma'$ where $\Star_w(\Sigma)=B(M_w)$ is the Bergman fan of some matroid $M_w$, which has $\sigma$ and $\sigma'$ as facets. Moreover, $B(M_w)$ is uniquely balanced with constant weights, and therefore since $w(\sigma)=1$ we must have $w(\sigma')=1$. By transitivity along the edges of $G(\Sigma)$, this implies that every facet in the same the connected component of $G(\Sigma)$ as $\sigma$ must have weight $1$. Since $G(\Sigma)$ is connected by \coloref{cor:connectedCodim1} we see that every facet $\sigma'$ must have weight $w(\sigma')=1$.
    
\end{proof}
If $\Sigma'\subsetneq \Sigma$ is a containment of tropical varieties of equal dimension, then, after choosing a common refinement on their polyhedral structure, there is a natural embedding of the top-dimensional Minkowski weights of $\Sigma'$ into $\Sigma$. In particular, if the space of Minkowski weights of $\Sigma$ is one-dimensional, then $\Sigma$ can have no proper tropical subvarieties of equal dimension.

\begin{corollary}
    Connected tropical manifolds have no proper tropical subvarieties of equal dimension.
\end{corollary}

\subsection{Dressians of Symmetric Powers}\label{subsec:SpacesofSymPow}
Before continuing to prove that the theory of linear tropical ideals is equivalent to the theory of large symmetric powers of matroids, we must establish that the space of symmetric powers of a fixed matroid is a compact subspace. For this, we will need to define valuated matroids through their Pl\"ucker vectors.

\begin{definition}\label{def:valuatedMatroid}
   A \textbf{valuated matroid} is a vector $\mu\in \TP^{\binom{[n]}{r}-1}$ satisfying the following \textbf{valued exchange axiom}. Let $B_1,B_2\in \binom{[n]}{r}$, then for each $b_1\in B_1\sm B_2$ there exists an $b_2\in B_2\sm B_1$ such that for any representative $\rho$ of $\mu$ we have
        \begin{equation}\label{eqn:basisExchange}
                \rho(B_1)\ttimes\rho(B_2)\geq \rho(B_1\sm b_1\cup b_2)\ttimes\rho(B_2\sm b_2\cup b_1).
        \end{equation}
    We note that if \coloref{eqn:basisExchange} is satisfied for some representative of $\mu$, then it is satisfied for every representative. A valuated matroid $\mu$, when presented as a vector in $\TP^{\binom{[n]}{r}-1}$ is called a \textbf{Pl\"ucker vector}.
    \end{definition}
The set of all Pl\"ucker vectors in $\TP^{\binom{n}{r}-1}$ is called the Dressian $\Dr(r,n)$, which is the moduli space of all valuated matroids of rank $r$ on $n$ elements. If the ground set $E$ of size $n$ is specified, we will write $\Dr(r,E)$ instead of $\Dr(r,n)$. Furthermore, we denote by $\Dr(M)$, the Dressian supported over $M$ or the Dressian of $M$, to be the set of valuated matroids whose underlying matroid is $M$. The Dressian $\Dr(r,n)$ is a closed polyhedral complex in $\TP^{\binom{n}{r}-1}$, moreover, since $\TP^{\binom{n}{r}-1}$ is compact, the Dressian is compact. However, $\Dr(M)$ is usually not compact, as it does not contain its boundary.

Let $M$ be a matroid with ground set $E$ and $T\subset E$. In order to specialize our focus to symmetric powers, we must understand how matroid restriction induces maps from $\Dr(M)$ to $\Dr(M|_T)$. Fix an independent set $I\subset E\sm T$ of rank $r-\rk(T)$ such that $T\cup I$ spans $E$. Given a valuated matroid $\sV\in \Dr(M)$, the restriction of $\sV$ to $T\subset E$ is given by
$$\sV|_T(B) = \sV(B\cup I).$$
The Pl\"ucker vector $\sV|_T\in \TP^{\binom{T}{\rk(T)}-1}$ does not depend on the choice of $I$ \cite{DW92}*{Proposition 1.2}. While the choice of $I$ does not matter when considering points of $\TP^{\binom{T}{\rk(T)}-1}$ we must consider the isometric embedding 
$$\imath_I:\TP^{\binom{T}{\rk(T)}-1}\hookrightarrow\TP^{\binom{[n]}{r}-1}$$ 
given by
 $$ (p_S)_{S\in\binom{T}{r}}\mapsto (a_B)_{B\in \binom{[n]}{r}} $$
 with
 $$a_B = \begin{cases}
 p_S & \text{if } B=S\cup I\\
 \infty & \text{if } B\neq S\cup I \text{ for any } S\subset [T].
 \end{cases}$$

 Given a Pl\"ucker vector $\sV\in \Dr(M)$, a restriction $\sV|_T\in \Dr(M|_{T})$, and a choice of $I$ with $T\cup I$ spanning $M$, we can interpret restriction as a coordinate projection 
 $$\pi_{T\cup I}: \TP^{\binom{n}{r}}\to \TP^{\binom{n}{r}}$$
 
 given by
$$\sV\mapsto \imath_I(\sV|_T).$$

\begin{proposition}\label{prop:plConvergence}
Let $\set{\sV_i}_{i=1}^\infty$ be a convergent sequence of Pl\"ucker vectors in $\Dr(r,n)\subset \TP^{\binom{n}{r}-1}$ with limit $\sV$. Then for any $T\subset [n]$ where $\sV|_T$ has positive rank, the sequence $\set{\sV_i|_T}_{i=1}^\infty$ converges to $\sV|_T$ in $\Dr(\rk(T), T)$
\end{proposition}
\begin{proof}
 Since $\Dr(r,n)$ is a finite union of the sets $\Dr(M)$, we may assume that the entire sequence $\set{\sV_i}_{i=1}^\infty$ lies in $\Dr(M)$ for fixed $M$, but possibly converges to a point in the boundary of $\Dr(M)$.  Since $\sV$ has at least one basis $B$, we choose to present our vectors such that $(\sV_i)_B = 0$. 
 Consider that for each $A\in \binom{[n]}{r}$ that 
 $$\lim_{i\to \infty} (\sV_i)_A = \sV_A.$$
 For any choice of independent set $I\subset [n]\sm T$ of rank $r-\rk(T)$ such that $T\cup I$ spans $M$ we have the following. Let $S\in \binom{T}{\rk(T)}$, the sequence
 $$\set{(\sV_i|_{T})_S}_{i=1}^\infty$$ 
 converges to $\sV_{S\cup I}$ as
 $$\lim_{i\to\infty}\set{(\sV_i|_{T})_S}_{i=1}^\infty=\lim_{i\to\infty}\set{(\sV_i)_{S\cup I}}_{i=1}^\infty=\sV_{S\cup I}.$$
 
Let $\rho\in \TP^{\binom{T}{\rk(T)}}$ be the vector with coordinates 
$$\rho_S = \sV_{S\cup I}.$$
We prove that there is a choice of independent set $I$ for which $T\cup I$ spans $\uMat(\sV)$ and that $I$ has rank $r-\rk(T)$ complementary rank in $\uMat(\sV)$, as given this we then have that $\rho_S = \imath\inv(\pi_{T\cup I}(\sV))= \sV|_T.$ Since $\rho_S = \sV_{S\cup I}$, we see that $\rho_S<\infty$ if and only if $S\cup I$ is a basis of $\uMat(\sV)$. Since $\rho$ has at least one basis then $\rho_S = \sV_{S\cup I}$ and $S\cup I$ is a basis of $\uMat(\sV)$. In particular $S\cup I\subset T\cup I$ spans $\uMat(\sV)$, $S$ has rank $\rk(T)$, and $I$ has rank $r-\rk(T)$, as desired. Therefore $\rho = \imath\inv(\pi_{T\cup I}(\sV)) = \sV|_T$, as desired.
\end{proof}

\begin{proposition}\label{prop:closedcompact}
Let $M$ be a fixed valuated matroid of rank $r$ in $\T^{[n]}$. Set
$$r_d:= \binom{r+d-1}{d}$$
$$n_d:= \binom{n+d-1}{d}.$$
Then the collection $\Dr_d(M)$ of $d\th$ symmetric powers of $M$ form a compact polyhedral complex in $\Dr(r_d,n_d)$.
\end{proposition}
\begin{proof}
Fix a valuated matroid $M$. We note that $M_d\in \Dr_d(M)$ determines $(M,M_2,\dots,M_d)$ uniquely, thus we only need to record the Pl\"ucker vector of $M_d$ in this construction, hence we work in $\TP^{\binom{n_d}{r_d}-1}$. 

It is sufficient to prove that the set of vectors in $\TP^{\binom{n_d}{r_d}-1}$ satisfying \coloref{eqn:symprodEmbedding} forms a closed polyhedral complex.

Let $p$ be the Pl\"ucker vector of $M_{d-1}$ and for each $\omega\in [n]$ fix a set $I_\omega$ such that $\omega\Sym_{d-1}(E)\cup I_\omega$ spans $\uMat(M_d)$. For simplicity, we abbreviate the associated coordinate projection by
$$\pi_\omega: = \pi|_{(\omega \Sym_{d-1}(E)\cup I_\omega)}.$$

The set of vectors in $\Dr(r_d,n_d)$ that satisfy \coloref{eqn:symprodEmbedding} for the embedding of $M_{d-1}$ by $\omega$ into $M_d$, for all $\omega\in [n]$, are those vectors in the intersection of the $Dr(r_d,n_d)$ and the following set
$$S:=\bigcap_{\omega\in [n]}\set{v\in \TP^{\binom{n_d}{r_d}-1}: \pi_\omega(v) = p}.$$

We now note that the set
$$S_\omega:= \set{v\in \TP^{\binom{n_d}{r_d}-1}: \pi_{\omega}(v) = p}$$

is a closed polyhedral subset of $\TP^{\binom{n_d}{r_d}}$, and therefore the finite intersection $S$ over all $S_\omega$ is closed and polyhedral. We conclude that $\Dr_d(M)=S\cap \Dr(r_d,n_d)$ is a compact polyhedral complex. 
\end{proof}

\subsection{Hilbert Functions}\label{subsec:HF}
    Now that we've established that a tropical linear space cannot have any proper tropical subvarieties of equal dimension, we can return to our discussion of linear tropical ideals. Recall that the \textbf{Hilbert function} of a homogeneous tropical ideal $I\subset \T[x_1,\dots,x_n]$ is given by
    $$H_I(d):= \rk(I_d).$$
    While this definition is given for homogeneous tropical ideals, it extends to a definition for ideals in $\T[x_0,\dots x_n]$ or $\T[x_1^{\pm1},\dots x_n^{\pm1}]$. See \cite{MR20}*{Lemma 2.1} for a comprehensive description. 
    
    The Hilbert function of a tropical ideal $I$ tells us valuable information about the dimension of the tropical variety $V(I)$, and this fact motivates what follows. In order to establish bounds on the value $H_I(d)$ we will need to define the independence complex of a tropcial ideal $I$.
\begin{definition}\label{def:independenceComplex}
    The \textbf{independence complex} of a tropical ideal $I\subset\T[x_1,\dots,x_n]$ is the set system on $[n]$ given by
    $$\sI(I):= \set{A\subset [n]: I\cap \T[x_i]_{i\in A} = \set{\infty}}$$
    or equivalently
    $$\sI(I):= \set{A\subset [n]: \pi_{A}(V(I))=\T^{A} }.$$
\end{definition}

\begin{proposition}\label{prop:linearHFequiv}
    Let $I\in \T[x_1,\dots,x_n]$ be a tropical ideal with $\rk(I_1)=r$. Then $I$ is linear if and only if
    $$H_I(d)=\binom{r+d-1}{d}.$$
\end{proposition}
To avoid confusion in the following proof, we note that our convention of $H_I(d)$ and that used in the paper \cite{DR19} differ. In this paper $H_I(d)=\rk(I_d)$, whereas in \cite{DR19} the convention $H_I(d)=\rk(I_{\leq d})$ is used.
\begin{proof}
Let $I$ be a linear tropical ideal with $\rk(I_1)=r$ and let $M$ be the underlying matroid $\uMat(I_1)$. Then by definition $V(I)=V(I_1)$ and $H_I(1)=\rk(I_1)=r$. By \cite{DR19}*{Proposition 5.5} we can immediately conclude that
$$H_I(d)\leq \binom{r+d-1}{d}.$$

To impose a lower bound, we notice that $\pi_A(V(I))=\T^A$ if and only if $\pi_A(B(M))=\T^A$. In particular, the independence complex $\sI(I_1)$ is equal to that of $B(M)$. \cite{Yu17}*{Lemma 3} then provides that $\sI(I_1)$ is the collection of independent sets of $M$, and therefore $\sI(I_1)$ contains a set of size $r$. By \cite{DR19}*{Proposition 5.4} we may conclude that
$$H_I(d)\geq \binom{r+d-1}{d}.$$ 
Combining these two inequalities provides that $H_I(d)= \binom{r+d-1}{d}$, as desired.

Conversely, suppose that $I$ has Hilbert function exactly $H_I(d)=\binom{r+d-1}{d}$. Then $I_1$ is a valuated matroid of rank $H_I(1)=r$ and $V(I_1)$ is a tropical linear space of dimension $r-1$. Moreover, the dimension of $V(I)$ is equal to the degree of the Hilbert polynomial of $H_I(d)$ \cite{MR20}*{Theorem 4.3}, which is a degree $r-1$ polynomial. We now note that $V(I)$ is a tropical subvariety of $V(I_1)$ which, by \coloref{cor:minkWeights} contains no proper tropical subvariety of dimension $r-1$. Thus we must have $V(I)=V(I_1)$, and therefore $I$ is linear.
\end{proof}

\begin{corollary}\label{cor:LTIsaturated}
If $I\subset \T[x_1,\dots,x_n]$ is a linear tropical ideal then

$$I = (\T[x_1^{\pm1},\dots,x_n^{\pm1}]\cdot I )\cap \T[x_1,\dots,x_n].$$
In other words, $I$ is saturated with respect to $\prod_{i=1}^n x_i$. 
\end{corollary}
\begin{proof}
Let $I$ be a linear tropical ideal, but suppose that towards a contradiction that $I$ is not saturated. Let $I_d$ be the minimal degree that is not saturated and let $f= x^\omega \cdot g\in I_d$ where $x^\omega$ is a degree $i$ monomial, $g$ is not divisible by any monomial, and $g\notin I_{d-i}$. This implies $x^\omega\cdot I_{d-i}\subsetneq I_d|_{x^\omega}$ are nested tropical linear spaces. In particular $I_d|_{x^\omega}$ is not full rank and therefore $$V(I_d|_{x^\omega})\subsetneq V(I_1).$$
However this immediately implies $V(I)\subsetneq V(I_1)$, a contradiction. Thus $I$ must be saturated.
\end{proof}

\begin{remark}\label{rem:bfCharacterisation}
The proofs of \cite{DR19}*{Proposition 5.4, Proposition 5.5}, which are used to prove \coloref{prop:linearHFequiv}, are made at the level of tropical linear spaces, and not tropical ideals. In particular, they can be restricted to any $d\th$ symmetric power of a matroid. Indeed, we may conclude $(M_1,\dots,M_d)$ is a symmetric power of $M$ if and only if $V(M_d)= \trop(M)$. 
\end{remark}

\begin{corollary}\label{cor:bfCharacterisation}
\coloritemref{prop:symprodCharacterisations}{bases} and \coloritemref{prop:symprodCharacterisations}{bf} are equivalent.
\end{corollary}

\begin{corollary}\label{cor:hfSymprodEquiv}
Let $I\subset \T[x_1,\dots,x_n]$ be a tropical ideal. Then $I$ is linear if and only if, for every $d\in \N$, $(I_1,\dots,I_d)$ is a $d\th$ symmetric power of $I_1$. 
\end{corollary}
\begin{proof}
    \coloref{cor:LTIsaturated} and \coloref{prop:linearHFequiv} imply that if $I\subset \T[x_1,\dots,x_n]$ is a linear tropical ideal then the sequences $(I_1,\dots,I_d)$ are $d\th$ symmetric powers of $I_1$. Moreover, if $I$ is a tropical ideal such that $(I_1,\dots,I_d)$ are $d\th$ symmetric powers of $I_1$, then 
    $$V(I) = \bigcap_{d=1}^\infty V(I_d),$$
    where, by \coloref{prop:symprodCharacterisations}, we have $V(I_d)=V(I_1)$. We conclude that
    $$V(I)=\bigcap_{d=1}^\infty V(I_1) = V(I_1),$$
    so $I$ is linear.
\end{proof}

We will use these result to understand how matroid symmetric powers tie in to the theory tropical ideals in a very natural way.

\subsection{Tropically Realizable Matroids}\label{subsec:TropRealizableMatroids}

Given a linear tropical ideal $I$, the data of its variety is encoded entirely in $I_1$, which is itself a valuated matroid, and we have seen that $I$ necessarily encodes infinitely large symmetric powers of $I_1$. We will now prove that these are in fact equivalent theories.

Recall that a tropical variety $X$ is said to be tropically realizable if there is a tropical ideal $I$ for which $V(I)= X$. We extended this terminology to matroids by saying a valuated matroid $\sV$ on $[n]$ is tropically realizable if there exists a linear tropical ideal $I$ in $\T[x_1,\dots,x_n]$ for which $V(I)=\trop(\sV)$, where $\trop(\sV)$ is the tropical linear space of $\sV$.

\begin{theorem}\label{thm:linearMatroidCorres}
Let $\sV$ be a valuated matroid. Then $\sV$ is tropically realizable if and only if for every $d$ some symmetric power $(\sV_1,\dots,\sV_d)$ exists.
\end{theorem}

\begin{proof}
Since \coloref{prop:plConvergence} requires $\sV$ to have positive rank we note that if $\sV$ is rank $0$ the proof is trivial as rank $0$ valuated matroids have unique symmetric powers. Thus we may assume $\sV$ has positive rank in the remainder of the proof.

Let $\sV$ be a tropically realizable matroid. Then there exists a tropical ideal $I\subset \T[x_1,\dots, x_n]$ with $V(I)=\trop(\sV)$ and $H_I(d)=\binom{\rk(I_1)+d-1}{d}$. \coloref{cor:hfSymprodEquiv} now provides that $(I_1,\dots,I_d)$ is a $d\th$ symmetric power of $I_1=\sV$ for all $d\in \N$. This proves the implication.

For the converse, suppose that we have an infinite collection of symmetric powers of $\sV$ that become arbitrarily large, say 
$$S = \set{(\sV=\sV_{1,d},\sV_{2,d},\dots, \sV_{d,d})}_{d=1}^\infty.$$
We claim that the existence of $S$ implies that there is an infinite sequence $J$ of symmetric powers of $\sV$
$$J = (\sV, \sV_2, \dots, \sV_d, \dots)$$
where now the matroids $\sV_i$ are fixed between all $d\th$ symmetric powers given in $J$. Given that $J$ exists we will see that $J$ defines a tropical ideal via the circuit polynomials of each $\sV_d$. Thus we construct $J$ given $S$.

Let $S_j$ be any subset of $S$. We denote by $S_{j,d}$ the multiset of Pl\"ucker vectors
$$\set{\rho_d\in \Dr_d(\sV): (\rho_1,\rho_2,\dots, \rho_d,\dots, \rho_k)\in S_j}.$$

Let $\sV_1 = \sV$ and $S_1=S$. We recursively define the valuated matroid $\sV_d$ and the subset $S_d\subset S$ as follows. Let $\sV_d$ be an accumulation point or any point of infinite multiplicity in $S_{d-1,d}$. Let $S_d$ be a subset of $S_{d-1}$ for which $\sV_d$ is the unique accumulation point or point of infinite multiplicity in $S_{d,d}$. We claim that for every $d\in \N$ that $\sV_d$ and $S_d$ exist and that $(\sV_1,\dots \sV_d)$ is a $d\th$ symmetric power of $\sV$.

\coloref{prop:closedcompact} provides that since $S_{d-1,d}$ is an infinite multiset contained in $\Dr_{d-1}(\sV)$ that $S_{d-1,d}$ must have an accumulation point or a point of infinite multiplicity. Fix $\sV_d$ to be any such point. By construction of $\sV_d$, there is necessarily an infinite subset $S_d\subset S_{d-1}$ such that $\sV_d$ is the unique accumulation point in $S_{d,d}$. 

We now claim that $(\sV_1,\dots,\sV_d)$ is a $d\th$ symmetric power of $\sV$. Let 
$$T:=\set{\rho_{d,i}}_{i=1}^\infty$$
be a convergent sequence of elements of $S_{d,d}$. By construction $T$ necessarily converges to $\sV_d$. \coloref{prop:plConvergence} then implies that for any $e \in E(\sV)$ that the sequence
$$T|_e:=\set{\rho_{d,i}|_{e}}_{i=1}^\infty$$
converges to $\sV_d|_e$. Since each $\rho_{d,i}$ is a $d\th$ symmetric power of $\sV$, \coloref{eqn:symprodEmbedding} provides that we have an equality of Pl\"ucker vectors. $\rho_{d,i}|_{e}=\rho_{d-1,i}$.
In particular, the sequence
$$T|_e = \set{\rho_{d-1,i}}_{i=1}^\infty\subset S_{d-1,d-1}.$$
By construction, any convergent sequence in $S_{d-1,d-1}$ necessarily converges to $\sV_{d-1}$. Thus $\sV_d|_e=\sV_{d-1}$. This implies that $(\sV_1,\dots,\sV_{d-1},\sV_d)$ is a symmetric power of $\sV_1=\sV$ if and only if $(\sV_1,\dots,\sV_{d-1})$ is a symmetric power. In the case where $d=2$ this follows immediately, and therefore holds for all $d$, as desired. We conclude that the infinite sequence $J=(\sV_1,\dots,\sV_d, \dots)$ exists.

We claim that for any sequence $J$ obtained by the above construction for any choice of accumulation points or points of infinite multiplicity in each step, that $J$ provides a linear tropical ideal. Consider that \coloref{eqn:symprodEmbedding} provides that the polynomials associated to each $\sV_d$ defined by $J$ form a saturated tropical ideal $I\subset \T[x_1,\dots,x_n]$ with Hilbert function $\binom{\rk(\sV)+d-1}{d}$. \coloref{prop:linearHFequiv} now provides that $I$ is a linear tropical ideal with $V(I)=V(I_1)=\trop(\sV)$, as desired.
\end{proof}

\begin{remark}
It remains an open problem whether or not there is a tropical linear space that is tropically realizable but not realizable over a field. 
\end{remark}

\subsection{Minor Closed}\label{subsec:minorsLinTropIdeals}
    In this section we will show how tropical geometry allows us to easily prove that the class of tropically realizable matroids is minor closed and that the class of matroids with $d\th$ symmetric power is minor closed.
    
\begin{theorem}\label{thm:minorClosed}
The class of tropically realizable matroids is minor closed. Specifically, given a tropical realization $I\subset \T[x_1,\dots,x_n]$ of the valuated matroid $\sV$, the contraction $\sV\contr n$ is tropically realized by the specialization $I_{x_n=\infty}$, and the deletion $\sV\sm n$ is tropically realized by $I\cap\T[x_1,\dots,x_{n-1}]$.
\end{theorem}
\begin{proof}
Without loss of generality, consider the ideal $J:= I|_{x_n=\infty}$. We claim $J$ tropically realizes $\sV\contr n$. We first prove that $V(J) = V(J_1)=\trop(\sV\contr n)$. Note that
$$V(J)= \pi_{[n-1]}(V(I)\cap \set{x: x_n = \infty}) = \pi_{[n-1]}(\trop(\sV)\cap\set{x: x_n=\infty}).$$
To see that $V(J)\subset \trop(\sV\contr n)$ it is sufficient to notice that $J_1 = \sV\contr n$. This holds as $I_1=\sV$ and $J_1=\sV\contr n$ by construction. Therefore $V(J)\subset V(J_1)\subset \trop(\sV\contr n)$. 

Now let $\omega\in \trop(\sV\contr n)$. We must show that every polynomial in $J$ attains its minimum twice on $\omega$. Note that $\omega\in \trop(\sV\contr n)$ if and only if $\omega'=(\omega,\infty)$ is in $\trop(\sV)=V(I)$. We now show that this implies $\omega\in V(J)$. Consider that every polynomial $g$ of $J$ is of the form $f|_{x_n=\infty}$ for some polynomial $f\in I$. In particular $f\in I$ attains its minimum twice or equals $\infty$ at $\omega'$ if and only if $g$ attains its minimum twice or equals $\infty$ at $\omega$. Since $\omega'\in V(I)$ it follows that $\omega\in V(J)$, as desired. 

Since $J_1=\sV/n$ and $V(J)=V(J_1)$, $J$ is linear and is a tropical representation of $\sV/n$, as desired.

As for the matroid $\sV\sm n$, consider the tropical ideal $J := I\cap \T[x_1,\dots,x_{n-1}]$. By \cite{MR20}*{Theorem 4.7} we have $V(J)=\pi_{[n-1]}(\trop(\sV))$, which is equal to $\trop(\sV\sm n)$. By construction $J_1$ is equal to $\sV\sm n$. Therefore $J$ tropically realizes $\sV\sm n$.
\end{proof}

\begin{lemma}\label{lem:geometricDeletion}
    If $\sV$ is a valuated matroid on $[n]$ and $\sV_d$ is a $d\th$ symmetric power of $\sV$. Then the valuated matroid
    $$P:=\sV_d\sm(n\cdot \Sym_{d-1}([n]))$$
    is a $d\th$ symmetric quasi power of $\sV\sm n$ and satisfies
    $$V(P) = \pi_{[n-1]}(\trop(\sV)).$$
    In particular, $P$ is a $d\th$ symmetric power of $\sV\sm n$.
\end{lemma}
\begin{proof} 
To show that $P$ is a symmetric quasi power we proceed by induction on $d$. Let $d=2$. Let $x\in [n-1]$. Then 
$$P|_x \cong \sV|_x\sm x\cdot n.$$
By \coloref{eqn:symprodEmbedding} and \coloref{rem:isoMinors}, we conclude $\sV|_x\sm x\cdot n \cong \sV\sm n$. Thus $P|_x\cong \sV\sm n$ for all $x\in [n-1]$ and $P$ is a second symmetric quasi product. Suppose for all $i<d$ that $P_i := \sV_i\sm (n\cdot \Sym_{d-1}([n])$ is an $i\th$ symmetric quasi power of $\sV\sm n$. Now let $\bfx \in \Sym_{d-i}([n])$ for $i\geq 1$. We have
\begin{align*}
    P|_{\bfx}   &= \sV|_{\bfx}\sm (\bfx\cdot n \cdot \Sym_{i-1}(E)) \\
                &= \sV_{i}\sm (n\cdot\Sym_{i-1}(E)\\
                &= P_i
\end{align*}
where this latter inequality is well defined by \coloref{rem:isoMinors}. In particular, $P$ is a symmetric quasi power of $P_1\cong \sV\sm n$, as desired. 

We now prove the latter claim. \coloref{prop:symprodCharacterisations} provides that $V(\sV_d)=\trop(\sV)$. Since $P$ is a symmetric quasi power of $\sV,$ $V(P)\subset \trop(V\sm n)$, which is equal to $\pi_{[n-1]}(\trop(\sV))$. It remains to prove the reverse inclusion. Let $w\in \pi_{[n-1]}(V(\sV_d))$. Then there is a $w' \in V(\sV_d)$ with $\pi_{[n-1]}(w')=w$. In particular, every circuit polynomial of $\sV_d$ attains its minimum twice or is equal to $\infty$ at $w'$. However, we note a circuit polynomial $f$ of $P$ is also circuit polynomial of $\sV_d$ and that $f$ attains its minimum twice or is equal to $\infty$ at $w'$ if and only if $f$ attains its minimum twice or is equal to $\infty$ at $w$. We conclude that $w\in V(P)$ so that $\pi_{[n-1]}(\trop(\sV))\subset V(P)$, as desired.
\end{proof}

\begin{corollary}\label{cor:minorClosed}
    The class of matroids with $d\th$ symmetric power is minor closed.
\end{corollary}
\begin{proof}
    The result follows from the argument used in \coloref{thm:minorClosed}, replacing \cite{MR20}*{Theorem 4.7} with \coloref{lem:geometricDeletion}.
\end{proof}

\begin{corollary}
    The class of tropically realizable matroids has infinitely many forbidden minors.
\end{corollary}
\begin{proof}
    This is an immediate consequence of \coloref{thm:sympowForbiddenMinors}.
\end{proof}

%%%%%%%%%%%%%%%%%%%%%%%%%%%%%%%%%%%%%%%%%%%%%%%%%%%%

%%%%%%%%%%%%%%%%%%%%%%%%%%%%%%%%%%%%%%%%%%%%%%%%%%%%

%%%%%%%%%%%%%%%%%%%%%%%%%%%%%%%%%%%%%%%%%%%%%%%%%%%%
\subsection{Tropically Algebraic and Set-realizable Matroids}\label{subsec:A&StRealizable}

    Matroids appear in many ways in the study of tropical ideals. In the previous sections we considered only linear tropical ideals, whose variety is a tropical linear space with weights $1$. However, we can just as easily ask about tropical linear spaces balanced with weights other than $1$. This remains largely an unexplored avenue of research. A valuated matroid $\sV$ is called \textbf{tropically set-realizable} if $\trop(\sV)$ is the support of $V(I)$ for some tropical ideal $I$. That is, we consider equality of $\trop(\sV)$ with $V(I)$ as sets with no polyhedral structure. See \cite{Yu17} for initial work in this direction. 
    
    Stemming from tropical set-realizability is also the study of \textbf{tropically algebraic matroids} --- matroids that appear as the independence complex of a tropical ideal, as in \coloref{def:independenceComplex}. Tropically algebraic matroids were first discussed in \cite{DR19} but there is yet any work exploring this topic.
    
    Restricting to the case of matroids rather than valuated matroids, we observe that the classes of tropically realizable, tropically set-realizable, and tropically algebraic matroids form nested classes of matroids in increasing order. There are yet any proofs that any of these containment relations is strict. However, in this section we conclude that all three classes are minor closed.

\begin{proposition}
The class of tropically set-realizable valuated matroids is minor closed.
\end{proposition}
\begin{proof}
The exact proof of \coloref{thm:minorClosed} applies here without the requirement that $I_1=M$, as the weights associated to the maximal cones are no longer considered.
\end{proof}

While minor-closedness of tropically set-realisable matroids follows the same proof as used for tropically realizable matroids, tropically algebraic matroids require a slightly different approach.

If the independence complex of $I$ forms the independent sets of a Boolean matroid $M$ we say that $I$ is a matroidal tropical ideal. Furthermore, if there is a tropical ideal with $\sI(I)=M$ then we say that $M$ is tropically algebraic. Algebraic matroids are tropically algebraic \cite{Yu17}*{Lemma 3}, but it is still open whether or not there is a tropically algebraic matroid that is not algebraic.

\begin{proposition}\label{prop:algMinorClosed}
The class of tropically algebraic matroids is minor closed.
\end{proposition}
\begin{proof}
Suppose $M$ is a tropically algebraic matroid, and let $I$ be a tropical ideal whose independence complex is $M$. Let without loss of generality we delete $n$. We claim that $J:=I\cap \T[x_1,\dots,x_{n-1}]$ has $\mathcal I(J) = M\sm n$. To see this, we first note \cite{MR20}*{Theorem 4.7} provides that
$$V(J)= \pi_{[n-1]}(V(I)).$$

A set $T$ is independent in $M\sm n$ if and only if $T$ is independent in $M$ and does not contain $n$. Thus we must show that $\pi_T(V(J))=\T^T$ if and only if $\pi_T(V(I))=\T^T$. To see this, we need only observe that $\pi_T(V(J))=\pi_T(\pi_{[n-1]}(V(I)))=\pi_T(V(I))$, as desired.

To prove contraction, if $n$ is a loop of $M$, then $M\contr n =M\sm n$, so we choose to delete $n$ instead and do the above. Otherwise $n$ is not a loop and we do the following. Consider the specialization $J:= I|_{x_n=0}$. The independence complex of $J$ is 
$$\mathcal I(J) = \set{T\subset [n-1] | J \cap \R[\set{x_i}_{i\in T}] = \set{\infty}}.$$

We claim $\mathcal I(J) = \mathcal I(M\contr n)$. In particular, we must show that $T\in \mathcal I(J)$ if and only if $T\cup i$ is independent in $\mathcal I(I)$. Since $\sI(J)$ is a simplicial complex, it is sufficient to prove this for its bases.

First let $B$ be a maximal subset of $\mathcal I(J)$. We claim $B\cup n$ is in $\mathcal I(I)$. Suppose not. Then there is a nonzero polynomial $f$ in $I\cap \T[x_i: i \in B\cup n]$. This implies that $f|_{x_n=0}\neq \infty$, where $f|_{x_n=0}\in J$, a contradiction. Thus $\mathcal I(J)\subset \mathcal I(M\contr n)$

That $\mathcal I(M\contr n)\subset \mathcal I(J)$ is immediate, as if $B\in \mathcal I(M\contr n)$ is a basis, we have that $B\cup \in \mathcal I(M)=\mathcal I(I)$, in particular, $I\cap \T[x_i: i\in (B\cup n)]=\set{\infty}$, and thus $J\cap \T[x_i: i\in B]=\set{\infty}$. This concludes the proof.
\end{proof}

In the proof of \coloref{prop:algMinorClosed}, we specialize our ideals at $x_n=0$, rather than $x_n=\infty$. However, in the case where $V(I)=\trop(\sV)$ for some valuated matroid $\sV$, both specialization at $x_n=0$ and $x_n=\infty$ yield tropical ideals that tropically algebraically represent $\mu\contr n$.

%%%%%%%%%%%%%%%%%%%%%%%%%%%%%%%%%%%%%%%%%%%%%%%%

%%%%%%%%%%%%%%%%%%%%%%%%%%%%%%%%%%%%%%%%%%%%%%%%

%%%%%%%%%%%%%%%%%%%%%%%%%%%%%%%%%%%%%%%%%%%%%%%%

%%%%%%%%%%%%%%%%%%%%%%%%%%%%%%%%%%%%%%%%%%%%%%%%

\bibliographystyle {alpha}
\bibliography{MatroidProducts}

\end{document}